\documentclass{amsart}

\usepackage[utf8]{inputenc}
\usepackage{amssymb}
\usepackage[leqno]{amsmath}
\usepackage{latexsym}
\usepackage{amsthm}
\usepackage{enumitem}
\usepackage{graphicx}
\usepackage[all,cmtip]{xy}
\usepackage{stmaryrd}
\usepackage{bigints}
\usepackage{esint}
\usepackage{ifthen}

\usepackage{xcolor}

\theoremstyle{plain}
\newtheorem{theorem}{Theorem}

\newtheorem{lemma}[theorem]{Lemma}
\newtheorem{proposition}[theorem]{Proposition}

\theoremstyle{remark}
\newtheorem{remark}[theorem]{Remark}

\numberwithin{theorem}{section}
\numberwithin{equation}{section}

\newcommand{\N}{\mathbb{N}}
\newcommand{\R}{\mathbb{R}}
\newcommand{\Cleq}{\ensuremath{\lesssim}}  
  
\newcommand{\Mesh}{\mathcal{M}}
\newcommand{\Submesh}{\mathcal{T}}
\newcommand{\Faces}[1]{\mathcal{F}_{#1}}

\newcommand{\FacesM}{\Faces{}}
\newcommand{\FacesMint}{\FacesM^i}
\newcommand{\FacesMbnd}{\FacesM^b}
\newcommand{\Shape}{\gamma}
\newcommand{\Normal}{n}
\newcommand{\Skeleton}{\Sigma}

\newcommand{\Lagr}[2]{\mathcal{L}_{#1}(#2)}
\newcommand{\LagrInt}[2]{\mathcal{L}^i_{#1}(#2)}
\newcommand{\Jump}[1]{\left \llbracket #1 \right \rrbracket}

\newcommand{\Domain}{\Omega}
\newcommand{\Dim}{d}
\newcommand{\Degree}{p}
\newcommand{\Poly}[2]{\mathbb{P}_{\ifthenelse{\equal{#1}{}}{\Degree} {#1}}(#2)}
\newcommand{\Leb}[1]{L^2(#1)}                                
\newcommand{\Sob}[1]{H^1(#1)}                              
\newcommand{\SobH}[1]{H^1_0(#1)}
\newcommand{\SobHD}[1]{H^{-1}(#1)}

\newcommand{\SpaceDisc}{S}
\newcommand{\SpaceCont}{V}
\newcommand{\DualSp}[1]{#1^*} 
\newcommand{\HHO}{\textup{{\tiny H}}}
\newcommand{\HHOspace}[1]{\hat S^{#1}_\HHO}  
\newcommand{\SpaceContHHO}{\hat \SpaceCont}   
\newcommand{\hs}{\hat s}
\newcommand{\hsigma}{\hat \sigma}
\newcommand{\hU}{\hat U}
\newcommand{\hu}{\hat u}
\newcommand{\hv}{\hat v}
\newcommand{\hw}{\hat w}
\newcommand{\hz}{\hat z}
\newcommand{\hZ}{\hat Z}
\newcommand{\hpsi}{\hat \psi}

\newcommand{\Grad}{\nabla}
\DeclareMathOperator{\GradM}{\nabla_{\Mesh}}             
\DeclareMathOperator{\Lapl}{\Delta}
\newcommand{\PotRec}{\mathcal{R}}
\newcommand{\PotRecSt}{\mathcal{S}}
\newcommand{\Interp}{\hat{\mathcal{I}}}
\newcommand{\EllPro}{\mathcal{E}}
\newcommand{\LebPro}{\Pi}
\newcommand{\BubbOper}[1]{\mathcal{B}_{#1}} 
\newcommand{\AvgOper}[1]{\mathcal{A}_{#1}}
\newcommand{\hAvgOper}[1]{\hat{\mathcal{A}_{#1}}}
\newcommand{\Smt}{E}
\newcommand{\HHOsmt}{{\Smt_{\HHO}}}
\newcommand{\hHHOsmt}{{\hat \Smt_{\HHO}}}
\newcommand{\id}{\mathrm{Id}}
\newcommand{\Ritz}{\hat{\mathcal{P}}}

\newcommand{\aext}{\widetilde{a}}
\newcommand{\aextHHO}{\aext_{\HHO}}
\newcommand{\HHObform}{b_{\HHO}}
\newcommand{\Stab}{\theta}   

\newcommand{\normSemi}[1]{\lvert #1 \rvert} 
\newcommand{\norm}[1]{\| #1 \|} 
\newcommand{\normHHO}[1]{\normSemi{#1}_{\aextHHO}}

\newcommand{\normLeb}[2]{\norm{#1}_{\Leb{#2}}}
\newcommand{\normOp}[3]{\norm{#1}_{\ifx#2#3\mathcal{L}(#2)\else\mathcal{L}(#2,#3)\fi}}       
\newcommand{\normDual}[2]{\left\| #1 \right\|_{\DualSp{#2}}}            

\begin{document}

\title[A quasi-optimal HHO method with $H^{-1}$ loads]{A quasi-optimal variant of the hybrid high-order method for elliptic PDEs with $H^{-1}$ loads}
 

\author[A.~Ern]{Alexandre Ern}
\address{Universit\'{e} Paris-Est, CERAMICS (ENPC), 77455 Marne-la-Vall\'{e}e cedex 2, France and INRIA Paris, 75589 Paris, France}
\email{alexandre.ern@enpc.fr}

\author[P.~Zanotti]{Pietro Zanotti}
\address{TU Dortmund \\ Fakult{\"a}t f{\"u}r Mathematik \\ D-44221 Dortmund \\ Germany}
\email{zanottipie@gmail.com}

\keywords{hybrid methods, quasi-optimality, rough loads, arbitrary order, general meshes}

\subjclass[2010]{65N30, 65N12, 65N15}

\begin{abstract}
Hybrid High-Order methods for elliptic diffusion problems have been originally formulated for loads in the Lebesgue space $\Leb{\Domain}$. In this paper we devise and analyze a variant thereof, which is defined for any load in the dual Sobolev space $\SobHD{\Domain}$. The main feature of the present variant is that its $H^1$-norm error can be bounded only in terms of the $H^1$-norm best error in a space of broken polynomials. We establish this estimate with the help of recent results on the quasi-optimality of nonconforming methods. We prove also an improved error bound in the $L^2$-norm by duality. Compared to previous works on quasi-optimal nonconforming methods, the main novelties are that Hybrid High-Order methods handle pairs of unknowns, and not a single function, and, more crucially, that these methods employ a reconstruction that is one polynomial degree higher than the discrete unknowns. The proposed modification affects only the formulation of the discrete right-hand side. This is obtained by properly mapping discrete test functions into $\SobH{\Domain}$.   
\end{abstract}

\maketitle

\section{Introduction}

Hybrid High-Order (HHO) methods have been introduced in \cite{DiPietro.Ern.Lemaire:14} for diffusion problems and in \cite{DiPietro.Ern:15} for locking-free linear elasticity. These methods employ face unknowns and cell unknowns and are devised from two local operators, a reconstruction operator and a stabilization operator. HHO methods support general meshes (with polyhedral cells and nonmatching interfaces), they are locally conservative and are robust in various regimes of practical interest, and they offer computational benefits resulting from the local elimination of cell unknowns by static condensation. The realm of applications of HHO methods has been vigorously expanded over the last few years; for brevity, we only mention \cite{BoDPS:2017,AbErPi2018,AbErPi2018a} for nonlinear solid mechanics and refer the reader to the bibliography therein. An Open-Source library for HHO methods based on generic programming is also available \cite{CicDE:2018}. Finally, we mention that HHO methods are closely related to hybridizable discontinuous Galerkin methods \cite{CoGoL:09} and to nonconforming virtual element methods \cite{Ayuso2016}, as shown in \cite{CoDPE:2016}. 

In the present work, we focus on the Poisson model problem which reads as follows:
\begin{equation}
	\label{Poisson}
	\begin{split}
		\text{Given } \; f \in \SobHD{\Domain}, 
		\; \text{ find } \; u \in \SobH{\Domain} \;\text{ such that }\\
		\forall w \in \SobH{\Domain} \quad \int_\Domain \Grad u \cdot \Grad w = \left\langle f, w\right\rangle_{\SobHD{\Domain} \times \SobH{\Domain}}.
	\end{split}
\end{equation}
Although the model problem~\eqref{Poisson} is posed for general loads in the dual Sobolev space $\SobHD{\Domain}$, the devising and analysis of HHO methods in \cite{DiPietro.Ern.Lemaire:14} requires that the load is in the Lebesgue space $\Leb{\Domain}$. In particular, $H^1$-norm error estimates with optimal decay rates have been derived in \cite{DiPietro.Ern.Lemaire:14} for smooth solutions in $H^{2+p}(\Domain)$ (where $p\ge0$ is the polynomial order of the face unknowns), and more generally hold true under the regularity requirement $u\in H^{1+s}(\Domain)$, $s>\frac12$, which is reasonable for the model problem~\eqref{Poisson} if $f\in \Leb{\Domain}$. Moreover, improved $L^2$-norm error estimates can also be derived by means of the Aubin--Nitsche duality argument. These results were recently extended in \cite{ErnGu:18} to the regularity requirement $u\in H^{1+s}(\Domain)$, $s>0$, and for loads in the Lebesgue space $L^q(\Domain)$ with $q>\frac{2d}{2+d}$, where $d$ is the space dimension, that is $q>1$ if $d=2$ and $q>\frac65$ if $d=3$. Therein, quasi-optimal error estimates were established in an augmented norm that is stronger than the $H^1$-norm. 

The above discussion shows that a theoretical gap still remains in the analysis of the HHO methods. One option to fill this gap would be to bound the $H^1$-norm error only in terms of the $H^1$-norm best error in the underlying discrete space. In fact, such quasi-optimal estimate would not require regularity assumptions beyond $\SobH{\Domain}$ for the solution and $\SobHD{\Domain}$ for the load. Notably, the abstract theory of \cite{Veeser.Zanotti:17p1} on the quasi-optimality of nonconforming methods indicates that an estimate in this form can be expected for a variant of the original HHO method of \cite{DiPietro.Ern.Lemaire:14}. This is the main achievement of the present work. In particular, the modified HHO method is defined and stable for arbitrary loads in $\SobHD{\Domain}$, as well as properly consistent.    

Quasi-optimality in the energy norm has been previously achieved by variants of classical nonconforming methods \cite{Veeser.Zanotti:17p2} and discontinuous Galerkin and other interior penalty methods \cite{Veeser.Zanotti:17p3} for second- and fourth-order elliptic problems.
Similarly to \cite{Veeser.Zanotti:17p2,Veeser.Zanotti:17p3}, our modification of the original HHO method affects only the discretization of the load. In the novel HHO method, the discrete test functions are transformed through an averaging operator to achieve stability and bubble smoothers to enforce consistency. 
The main novelties concerning the analysis of nonconforming quasi-optimal methods is that we extend the abstract framework of \cite{Veeser.Zanotti:17p1} so as to handle pairs of functions (one defined in the computational domain and one on the mesh skeleton) and that we deal with the presence of a reconstruction operator that is one polynomial degree higher than the discrete unknowns. In addition to quasi-optimality in the $H^1$-norm, we also show that the Aubin--Nitsche duality argument still allows one to derive improved $L^2$-norm error estimates for the modified HHO method. Finally, owing to the results from \cite{CoDPE:2016}, we notice that the present findings thus provide a way to achieve the same quasi-optimal properties by appropriately modifying the discrete load in hybridizable discontinuous Galerkin and nonconforming virtual element methods.

This paper is organized as follows. In section~\ref{S:abstract-framework}, we briefly 
summarize the abstract framework from \cite{Veeser.Zanotti:17p1} for 
quasi-optimal nonconforming methods. In section~\ref{S:HHO-method} we outline the 
main ideas and results concerning HHO methods on simplicial meshes with loads 
in $\Leb{\Domain}$. In section~\ref{S:quasi-optimal-HHO},
we present and analyze a quasi-optimal variant of the HHO method. This section
contains the main results of this work. Finally, in section~\ref{S:polytopic-meshes},
we show how the results of section~\ref{S:quasi-optimal-HHO} can be extended 
to the setting of polytopic meshes.


\section{Abstract framework for quasi-optimality}
\label{S:abstract-framework}

In this section we summarize the framework of \cite{Veeser.Zanotti:17p1} in a form that is convenient to guide the design of the method proposed in section~\ref{S:quasi-optimal-HHO}. Moreover, we recall the notion of quasi-optimality and a couple of related results. 

Let $\SpaceCont$ be a Hilbert space with scalar product $a$. Denote by $\DualSp{\SpaceCont}$ the topological dual space of $\SpaceCont$ and consider the elliptic variational problem
\begin{equation}
\label{cont-prob-abstract}
\text{Given } \; \ell \in \DualSp{\SpaceCont}, 
\; \text{ find } \; u \in \SpaceCont \;\text{ such that }\;
\forall w \in \SpaceCont \quad a(u,w) = \left\langle \ell, w\right\rangle_{\DualSp{\SpaceCont} \times \SpaceCont}
\end{equation}
which is uniquely solvable, according to the Riesz representation theorem.

Let $\SpaceDisc$ be a finite-dimensional linear space and assume that $a$ can be extended to a scalar product $\aext$ on $\SpaceCont + \SpaceDisc$, inducing the extended \textit{energy norm} $\norm{\cdot} := \sqrt{\aext(\cdot,\cdot)}$. Let $\Smt: \SpaceDisc \to \SpaceCont$ be a linear operator and consider the following approximation method for \eqref{cont-prob-abstract}:
\begin{equation}
\label{disc-prob-abstract}
\text{Given } \; \ell \in \DualSp{\SpaceCont}, 
\; \text{ find } \; U \in \SpaceDisc \;\text{ such that }\;
\forall \sigma \in \SpaceDisc \quad \aext(U,\sigma) = \left\langle \ell, \Smt \sigma\right\rangle_{\DualSp{\SpaceCont} \times \SpaceCont}
\end{equation} 
which is uniquely solvable, due to the positive-definiteness of $\aext$ on $\SpaceDisc$. We say that $\SpaceDisc$ is a \textit{nonconforming space} and \eqref{disc-prob-abstract} is \textit{nonconforming method} whenever $\SpaceDisc \nsubseteq \SpaceCont$.

Since $U \in \SpaceDisc$, the approximation error $u-U$ satisfies $\inf_{s \in \SpaceDisc} \norm{u-s} \leq \norm{u-U}$, showing that the best error $\inf_{s \in \SpaceDisc} \norm{u-s}$ is an intrinsic benchmark for \eqref{disc-prob-abstract}. Hence, we say that the method \eqref{disc-prob-abstract} is \textit{quasi-optimal} for \eqref{cont-prob-abstract} in the norm 
$\norm{\cdot}$ if there is a constant $C\geq 1$ such that
\begin{equation}
\label{quasi-optimality}
\norm{u-U} 
\leq
C \inf \limits_{s \in \SpaceDisc} \norm{u-s}
\end{equation}  
and $C$ is independent of $u$ and $U$. In this case, we refer to the best value of $C$ as the \textit{quasi-optimality constant} of \eqref{disc-prob-abstract} in the norm $\norm{\cdot}$.

\begin{remark}[Smoothing and stability by $\Smt$]
	\label{R:smoothing}
	We call $\Smt$ \textit{smoother}, because its action often increases the regularity of the elements of $\SpaceDisc$. An immediate observation is that the use of a smoother makes the duality $\left\langle \ell, \Smt \sigma \right\rangle _{\DualSp{\SpaceCont} \times \SpaceCont}$ in \eqref{disc-prob-abstract} well-defined for all $\ell \in \DualSp{\SpaceCont}$, irrespective of the possible nonconformity of $\SpaceDisc$. Notice also that $\Smt$ is bounded, because $\SpaceDisc$ is finite-dimensional. Thus, we infer that \eqref{disc-prob-abstract} is a stable method, in that
	\begin{equation}
	\label{full-stability}
	\norm{U} 
	\leq
	\normOp{\Smt}{\SpaceDisc}{\SpaceCont} \normDual{\ell}{\SpaceCont}
	=
	\normOp{\Smt}{\SpaceDisc}{\SpaceCont} \norm{u}.
	\end{equation} 
	Moreover, the operator norm of $\Smt$ is the best possible constant in this inequality for an arbitrary load $\ell \in \DualSp{\SpaceCont}$. The above stability \eqref{full-stability} is necessary for quasi-optimality, owing to the triangle inequality, and the quasi-optimality constant can be bounded from below in terms of the operator norm of $\Smt$.
\end{remark}

\begin{remark}[Feasible smoothers]
	\label{R:feasible-smoothing}
	The computation of $U$ from \eqref{disc-prob-abstract} requires the evaluation of $\Smt$ on each element of the basis $\{ \phi_1, \dots, \phi_n \}$ of $\SpaceDisc$ at hand. Hence, it is highly desirable that the duality $\left\langle \ell, \Smt \phi_i\right\rangle_{\DualSp{V} \times V}$, $i=1,\dots,n$,
	can be evaluated with $\mathrm{O(1)}$ operations. To this end, a sufficient condition is that each $\Smt \phi_i$ is locally supported and can be obtained from $\phi_i$ with $\mathrm{O}(1)$ operations.  
\end{remark}

Conforming Galerkin methods for \eqref{cont-prob-abstract} fit into this abstract framework with
\begin{equation*}
\label{conforming-Galerkin}
\SpaceDisc \subseteq \SpaceCont
\qquad \qquad
\aext = a
\qquad \qquad
\Smt = \id 
\end{equation*}
and are quasi-optimal in the energy norm, according to the so-called C\'{e}a's lemma \cite{Cea:64}. 
%
Still, quasi-optimality can be achieved also if $\SpaceDisc$ is nonconforming, depending on the interplay of $\aext$ and $\Smt$. In fact, \cite[Theorem~4.14]{Veeser.Zanotti:17p1} states that the following \textit{algebraic consistency}
\begin{equation}
\label{full-consistency}
\forall s \in \SpaceCont \cap \SpaceDisc,
\;\sigma \in \SpaceDisc
\qquad
\aext(s, \Smt \sigma)
=
\aext(s, \sigma).
\end{equation}
is necessary and sufficient for the existence of a constant $C$ so that \eqref{quasi-optimality} holds. This is actually equivalent to prescribe that the solution $u$ of \eqref{cont-prob-abstract} solves also \eqref{disc-prob-abstract}, whenever $u \in \SpaceCont \cap \SpaceDisc$. 

It is worth noticing that \eqref{full-consistency} is, possibly, a mild or trivial condition, as it involves only conforming trial functions $s \in \SpaceCont \cap \SpaceDisc$. Thus, it is not a surprise that additional informations are actually needed in order to access the size of the quasi-optimality constant. For instance, Remark~\ref{R:smoothing} reveals a lower bound in terms of the operator norm of the employed smoother. Moreover, one may expect that the quasi-optimality constant depends also on the discrepancy of the left- and right-hand sides of \eqref{full-consistency} for nonconforming trial functions $s \in \SpaceDisc \setminus \SpaceCont$. This claim can be confirmed with the help of \cite[Theorem~4.19]{Veeser.Zanotti:17p1}.    

In section~\ref{S:quasi-optimal-HHO}, we actually build on a generalized version of \eqref{full-consistency}, because the setting considered there does not exactly fit into the framework described here. 


\section{The HHO method on simplicial meshes}
\label{S:HHO-method}

In this section we recall the HHO method for \eqref{Poisson} proposed in \cite{DiPietro.Ern.Lemaire:14} and some of its properties. In order to avoid unnecessary technicalities, we restrict our attention to matching simplicial meshes, here and in the next section. The extension of our results to more general meshes is addressed in section~\ref{S:polytopic-meshes}.

\subsection{Discrete problem}
\label{SS:discrete-problem}

%
Let $\Domain \subseteq \R^\Dim$, $\Dim \in \{2,3\}$, be an open and bounded polygonal/polyhedral set with Lipschitz-continuous boundary. Let $\Mesh = (K)_{K \in \Mesh}$ be a matching simplicial mesh of $\Domain$, i.e., all cells of $\Mesh$ are $\Dim$-simplices and, for any $K \in \Mesh$ with vertices $\{ a_0, \dots, a_\Dim \}$ and for all $K' \in \Mesh$, the intersection $K \cap K'$ is either empty or the convex hull of a subset of $\{a_0, \dots, a_\Dim\}$.

We denote by $\FacesMint$ the set of all interfaces of $\Mesh$. Since the mesh is matching, any interface $F \in \FacesMint$ is such that $F = K_1 \cap K_2$ for some $K_1, K_2 \in \Mesh$, and $F$ is a full face of both $K_1$ and $K_2$. Similarly, we collect the boundary faces into $\FacesMbnd$ and observe that each $F \in \FacesMbnd$ satisfies $F = K \cap \partial \Domain$ for some $K \in \Mesh$. Then, the set $\FacesM := \FacesMint \cup \FacesMbnd$ consists of all faces, and the skeleton of $\Mesh$ is given by 
\begin{equation*}
	\label{skeleton}
	\Skeleton := \bigcup_{F \in \FacesM} F.
\end{equation*}

For each $K \in \Mesh$, we denote by $\Faces{K}$ the set of all faces of $K$, i.e. the faces $F \in \FacesM$ such that $F \subseteq K$. We indicate by $h_K$ and $h_F$ the diameters of $K$ and $F$, respectively. Moreover, we write $\Normal_K$ for the outer normal unit vector of $K$.

For $\Degree \in \N_0 := \N \cup \{0\}$, $K \in \Mesh$ and $F \in \FacesM$, let $\Poly{\Degree}{K}$ and $\Poly{\Degree}{F}$ be the spaces of all polynomials of total degree $\leq \Degree$ in $K$ and $F$, respectively. The corresponding broken spaces on $\Mesh$ and $\Sigma$ are given by
\begin{align*}
	\Poly{\Degree}{\Mesh} &:= \{ s_\Mesh: \Domain \to \R \mid \forall K \in \Mesh \;\; (s_\Mesh)_{|K} \in \Poly{\Degree}{K} \},	\\[1mm]
	\Poly{\Degree}{\Skeleton} &:= \{s_\Skeleton: \Skeleton \to \R \mid \forall F \in \FacesM \;\; (s_\Skeleton)_{|F} \in \Poly{\Degree}{F}, \quad \forall F \in \FacesMbnd \;\; (s_\Skeleton)_{|F} = 0 \}.
\end{align*}

We shall make use of the $L^2$-orthogonal projections $\LebPro_\Mesh : \Leb{\Domain} \to \Poly{\Degree}{\Mesh}$ and $\LebPro_\Skeleton: \Leb{\Skeleton} \to \Poly{\Degree}{\Skeleton}$, which are defined such that
\begin{equation*}
\label{L2-projections}
\int_K q \LebPro_\Mesh v_\Mesh = \int_K q v_\Mesh
\qquad \text{and} \qquad
\int_F r \LebPro_\Skeleton v_\Skeleton = \int_F r v_\Skeleton
\end{equation*}
for all $K \in \Mesh$, $q \in \Poly{\Degree}{K}$, $v_\Mesh \in \Leb{\Domain}$ and for all $F \in \FacesMint$, $r \in \Poly{\Degree}{F}$, $v_\Skeleton \in \Leb{\Skeleton}$, respectively.  

The HHO space of degree $\Degree$ is the Cartesian product
\begin{equation*}
	\label{HHO-space}
	\HHOspace{\Degree}
	:=
	\Poly{\Degree}{\Mesh} \times \Poly{\Degree}{\Skeleton},
\end{equation*}
so that the elements of $\HHOspace{\Degree}$ are pairs $\hs = (s_\Mesh, s_\Skeleton)$. The first component of $\hs$ is intended to approximate the solution $u$ of \eqref{Poisson} in each simplex of $\Mesh$, whereas the second component is intended to approximate the trace of $u$ on each face composing the skeleton $\Sigma$. Notice that the face component of a member of $\HHOspace{\Degree}$ incorporates the boundary condition of \eqref{Poisson}. In what follows, we denote both pairs and spaces of pairs using a hat symbol. Moreover, we drop the superscript $p$ and simply write $\HHOspace{}$ to alleviate the notation. The subscript 'H' serves to distinguish the HHO space from its abstract counterpart in section~\ref{S:abstract-framework}.

The first constitutive ingredient of the HHO method is a suitable \textit{higher-order} reconstruction. This is realized through the linear operator $\PotRec: \HHOspace{} \to \Poly{\Degree+1}{\Mesh}$, which is uniquely determined by the conditions
\begin{subequations}
	\label{gradient-reconstruction}
	\begin{equation}
		\label{gradient-reconstruction-grad}
		\forall q \in \Poly{\Degree+1}{K}
		\;\;
		\int_K \Grad \PotRec \hs \cdot \Grad q
		=
		-\int_K s_{\Mesh} \Lapl q
		+
		\int_{\partial K} s_\Skeleton  \Grad q \cdot \Normal_K 
	\end{equation} 
	and
	\begin{equation}
		\label{gradient-reconstruction-avg}
		\int_K \PotRec \hs 
		=
		\int_K s_\Mesh
	\end{equation}  
\end{subequations} 
for all $K \in \Mesh$ and $\hs=(s_\Mesh, s_\Skeleton) \in \HHOspace{}$. The local problem in \eqref{gradient-reconstruction-grad} is uniquely solvable up to an additive constant, which is then fixed by \eqref{gradient-reconstruction-avg}. The computation of $\PotRec \hs$ can be performed element-wise because, for each $K \in \Mesh$, the restriction $(\PotRec \hs)_{|K}$ depends only on $(s_\Mesh)_{|K}$ and $(s_\Skeleton)_{|\partial K}$.

The second constitutive ingredient of the HHO method is the following stabilization bilinear form defined on $\HHOspace{} \times \HHOspace{}$:
\begin{equation}
	\label{HHO-stabilization}
	\Stab(\hs,\hsigma)
	:=
	\sum \limits_{K \in \Mesh}
	\sum \limits_{F \in \Faces{K}}
	\int_F  
	\LebPro_\Sigma (s_\Skeleton - (\PotRecSt \hs)_{|K})
	\LebPro_\Sigma (\sigma_\Skeleton - (\PotRecSt \hsigma)_{|K}),
\end{equation}
with arbitrary $\hs = (\hs_\Mesh, \hs_\Skeleton)$ and $\hsigma = (\hsigma_\Mesh, \hsigma_\Skeleton)$ in $\HHOspace{}$ and with the stabilization operator $\PotRecSt: \HHOspace{} \to \Poly{\Degree+1}{\Mesh}$ such that
\begin{equation}
	\label{potential-reconstruction}
	\PotRecSt \hs 
	:=
	s_\Mesh + (\id - \LebPro_\Mesh) \PotRec \hs.
\end{equation}
Since both $\PotRec$ and $\LebPro_\Mesh$ can be computed element-wise, the operator $\PotRecSt$ inherits this property.

Denote by $\GradM$ the broken gradient on $\Mesh$, whose action on an element-wise $H^1$-function $v$ is given by $(\GradM v)_{|K} := \Grad (v_{|K})$ for all $K \in \Mesh$. The HHO bilinear form on $\HHOspace{}\times \HHOspace{}$ can be written as follows:
\begin{equation*}
	\label{HHO-bilinear-form}
	\HHObform (\hs,\hsigma)
	:=
	\int_\Domain \GradM \PotRec \hs \cdot \GradM \PotRec \hsigma
	+
	\Stab(\hs,\hsigma).
\end{equation*}
It can be verified that, for all $\hs \in \HHOspace{}$, the semi-norm $\Stab(\hs,\hs)^{\frac{1}{2}}$ penalizes the discrepancy between the face component of $\hs$ and the trace of the cell component on the skeleton $\Skeleton$. This, in turn, enforces positive-definiteness of $\HHObform$, as stated in Lemma~\ref{L:discrete-stability} below.

Assume for the moment that the load $f$ of \eqref{Poisson} is in $\Leb{\Domain}$. The HHO method of \cite{DiPietro.Ern.Lemaire:14} for the Poisson problem reads
\begin{equation}
	\label{HHO-method-standard}
	\text{Find } \; \hU \in \HHOspace{} 
	\;\text{ such that }\;
	\forall \hsigma = (\sigma_\Mesh, \sigma_\Skeleton) \in \HHOspace{}
	\quad
	\HHObform(\hU, \hsigma)
	=
	\int_\Domain f \sigma_\Mesh.   
\end{equation}


\subsection{Discrete stability and approximation properties}
\label{SS:stability-interpolation}

We now aim at assessing the stability of the form $\HHObform$ and the approximation properties of the space $\HHOspace{}$. 
For all $K \in \Mesh$, denote by $r_K$ the radius of the largest ball inscribed in $K$. The shape parameter $\Shape = \Shape(\Mesh)$ of the mesh $\Mesh$ is defined as the largest positive real number such that 
\begin{equation}
	\label{shape-parameter}
	\forall K \in \Mesh \qquad
	\Shape \,r_K \leq h_K.
\end{equation}
%
We indicate by $C_\Shape$ and $C_{\Shape, \Degree}$ two generic functions of the quantities indicated by the subscripts, nondecreasing in each argument, which do not need to be the same at each occurrence. Sometimes, we use the abbreviation $A \Cleq B$ in place of $A \leq C_{\Shape, \Degree} B$. 

For instance, if $K \in \Mesh$ and $F \in \Faces{K}$, the so-called discrete and continuous trace inequalities read
\begin{align}
	\label{discrete-trace-inequality}
	&\forall q \in \Poly{\Degree+1}{K}
	\quad \;
	h_F^{-\frac{1}{2}} \normLeb{q}{F}
	\leq
	C_{\Shape, \Degree}
	h_K^{-1}\normLeb{q}{K}\\
	\label{trace-inequality}
	&\forall v \in \Sob{K} \qquad
	h_F^{-\frac{1}{2}} \normLeb{v}{F}
	\leq
	C_\Shape
	\left( h_K^{-1} \normLeb{v}{K}
	+
	\normLeb{\Grad v}{K}\right) 
\end{align}
see, e.g., \cite[Lemmata~1.46~and~1.49]{DiPietro.Ern:12}. Recall also that, if $v \in \Sob{K}$ and $\int_K v = 0$, we have the Poincar\'{e}--Steklov inequality \cite{Bebendorf:03}
\begin{equation}
	\label{poincare-inequality}
	\normLeb{v}{K}
	\leq
	\pi^{-1} h_K \normLeb{\Grad v}{K}.
\end{equation}

The following result implies that the HHO bilinear form $\HHObform$ is nondegenerate and ensures that the problem \eqref{HHO-method-standard} is uniquely solvable. A proof can be found in \cite[Lemma~4]{DiPietro.Ern.Lemaire:14}.
\begin{lemma}[Coercivity of $\HHObform$]
	\label{L:discrete-stability}
	For all $\hs = (s_\Mesh, s_\Skeleton) \in \HHOspace{}$, we have
	\begin{equation*}
		\label{discrete-stability}
		\norm{\hs}_{\HHOspace{}}^2
		\leq
		C_{\Shape, \Degree}
		\HHObform(\hs,\hs),
	\end{equation*}
	where the norm $\norm{\cdot}_{\HHOspace{}}$ is defined as
	\begin{equation}
		\label{norm-coercivity}
		\norm{\hs}_{\HHOspace{}}^2
		:=
		\sum \limits_{K \in \Mesh} \left( 
		\normLeb{\Grad s_\Mesh}{K}^2
		+
		\sum \limits_{F \in \Faces{K}}
		h_F^{-1} \normLeb{s_\Sigma - (s_\Mesh)_{|K}}{F}^2 \right). 
	\end{equation}
\end{lemma}

Next, we examine the approximation properties of the HHO space underlying \eqref{HHO-method-standard}. To this end, we consider the interpolant $\Interp: \SobH{\Domain} \to \HHOspace{}$ defined as follows:
\begin{equation}
	\label{HHO-interpolant}
	\Interp v
	:=
	(\LebPro_\Mesh v, \LebPro_\Skeleton v).
\end{equation}

As the reconstruction $\PotRec$ maps the elements of $\HHOspace{}$ into piecewise polynomials of degree $(\Degree+1)$, we compare, in particular, the approximation in the mapped space $\PotRec (\HHOspace{})$ with the one in the broken space $\Poly{\Degree+1}{\Mesh}$, with respect to the $H^1$-norm and the $L^2$-norm. For this purpose, we make use of the broken elliptic projection $\EllPro: \SobH{\Domain} \to \Poly{\Degree+1}{\Mesh}$, which is obtained by imposing
\begin{subequations}
	\label{elliptic-projection}
	\begin{equation}
		\label{elliptic-projection-grad}
		\forall q \in \Poly{\Degree+1}{K}
		\qquad
		\int_K \Grad \EllPro v \cdot \Grad q
		=
		\int_K \Grad v \cdot \Grad q
	\end{equation}
	and
	\begin{equation}
		\label{elliptic-projection-avg}
		\int_K \EllPro v
		=
		\int_K v
	\end{equation}
\end{subequations}
for all $K \in \Mesh$ and $v \in \SobH{\Domain}$. 

Recall the definitions of $\PotRec$ and $\Interp$ from \eqref{gradient-reconstruction} and \eqref{HHO-interpolant}, respectively, and let $v \in \SobH{\Domain}$ be given. We have $\int_K \Grad \PotRec \Interp v \cdot \Grad q = -\int_K \LebPro_\Mesh v \Lapl q + \int_{\partial K} \LebPro_\Sigma v \Grad q \cdot \Normal_K = \int_K \Grad v \cdot \Grad q$ for all $K \in \Mesh$ and $q \in \Poly{\Degree+1}{K}$. Furthermore, $\int_K \PotRec \Interp v = \int_K \LebPro_\Mesh v = \int_K v$. Then, comparing with \eqref{potential-reconstruction} and \eqref{elliptic-projection}, we derive the identities
\begin{equation}
	\label{RI=E}
	\PotRec \circ \Interp = \EllPro
	\qquad \text{ and } \qquad
	\PotRecSt \circ \Interp = \EllPro + \LebPro_\Mesh (\id - \EllPro)
\end{equation} 
which can be used to assess the approximation properties of $\HHOspace{}$. 

\begin{lemma}[Interpolation errors]
	\label{L:interpolation-errors}
	For all $v \in \SobH{\Domain}$, we have
	\begin{subequations}
		\label{interpolation-errors}
		\begin{equation}
			\label{interpolation-errors-H1-norm}
			\normLeb{\GradM (v-\PotRec \Interp v)}{\Domain}^2
			+
			\Stab(\Interp v,\Interp v)
			\Cleq
			\sum \limits_{K \in \Mesh}
			\inf \limits_{q \in \Poly{\Degree+1}{K}}
			\normLeb{\Grad(v-q)}{K}^2,
		\end{equation}
		\begin{equation}
			\label{interpolation-errors-L2}
			\normLeb{v-\PotRec \Interp v}{\Domain}^2
			\leq 
			\sum \limits_{K \in \Mesh}
			\left( \dfrac{h_K}{\pi}\right)^2  \inf \limits_{q \in \Poly{\Degree+1}{K}}
			\normLeb{\Grad (v-q)}{K}^2.
		\end{equation}
	\end{subequations}
\end{lemma}
\begin{proof}
The proof follows from \cite{DiPietro.Ern.Lemaire:14} and is briefly sketched for completeness.
Let $v \in \SobH{\Domain}$ be given. The first summand in the left-hand side of \eqref{interpolation-errors-H1-norm} can be rewritten using the first part of \eqref{RI=E} and the element-wise $H^1$-orthogonality of $\EllPro$, which imply that
\begin{equation}
\label{elliptic-projection-orthogonality}
\normLeb{\Grad(v-\PotRec \Interp v)}{K}
=
\normLeb{\Grad(v-\EllPro v)}{K}
=
\inf_{q \in \Poly{\Degree+1}{K}} 
\normLeb{\Grad(v-q)}{K}
\end{equation}
for all $K \in \Mesh$. Concerning the other summand, the second part of \eqref{RI=E} reveals $\PotRecSt \Interp v = \EllPro v + \LebPro_\Mesh(v - \EllPro v)$. Inserting this identity into \eqref{HHO-stabilization}, we infer that 
\begin{equation}
	\label{stabilization+interpolation}
	\Stab(\Interp v,\Interp v)
	=
	\sum \limits_{K \in \Mesh} \sum \limits_{F \in \Faces{K}}
	h_F^{-1} \normLeb{\LebPro_\Skeleton(v- (\EllPro v)_{|K}) + \LebPro_\Mesh(v- \EllPro v)_{|K}}{F}^2.
\end{equation}
Consider any $K \in \Mesh$ and $F \in \Faces{K}$. We exploit the boundedness of $\LebPro_\Skeleton$ in the $L^2(F)$-norm, the trace inequality \eqref{trace-inequality}, the identity \eqref{elliptic-projection-avg}, and the Poincar\'{e}--Steklov inequality \eqref{poincare-inequality} to infer that
\begin{equation*}
h_F^{-\frac{1}{2}} \normLeb{\LebPro_\Skeleton(v- (\EllPro v)_{|K})}{F}
\leq C_\Shape
\normLeb{\Grad (v - \EllPro v)}{K}.
\end{equation*}
Next, we invoke the discrete trace inequality \eqref{discrete-trace-inequality}, the boundedness of $\LebPro_\Mesh$ in the $L^2(K)$-norm, the identity \eqref{elliptic-projection-avg}, and the Poincar\'{e}--Steklov inequality \eqref{poincare-inequality} to obtain
\begin{equation*}
h_F^{-\frac{1}{2}} \normLeb{\LebPro_\Mesh(v-\EllPro v)_{|K}}{F}
\leq C_{\Shape, p}
\normLeb{\Grad (v - \EllPro v)}{K}.
\end{equation*}
Hence, we derive the claimed bound \eqref{interpolation-errors-H1-norm} inserting these estimates into \eqref{stabilization+interpolation} and using again \eqref{elliptic-projection-orthogonality}. Finally, the identity \eqref{elliptic-projection-avg}, the first identity in \eqref{RI=E}, and the Poincar\'{e}--Steklov inequality \eqref{poincare-inequality} yield
\begin{equation*}
	\label{best-errors-proof}
	\normLeb{v - \PotRec \Interp v}{\Domain}^2
	\leq 
	\sum \limits_{K \in \Mesh} \left( \dfrac{h_K}{\pi}\right)^2  \normLeb{\Grad (v- \EllPro v)}{K}^2.
\end{equation*}
This proves \eqref{interpolation-errors-L2}, in combination with \eqref{elliptic-projection-orthogonality}.
\end{proof}

The first estimate in Lemma~\ref{L:interpolation-errors} has the remarkable property that the left- and the right-hand sides are equivalent, according to the inclusion $\PotRec \Interp v \in \Poly{\Degree+1}{\Mesh}$. The other estimate does not enjoy the same property, because the right-hand side requires higher regularity of $v$ than the left-hand side. Note also that, in both estimates, the right-hand side involves only local best errors on the simplices of $\Mesh$, in the spirit of \cite{Veeser:16}. This entails that the best approximation in $\HHOspace{}$ needs only piecewise (and not global) regularity of $v$ to achieve convergence with a certain decay rate. 

Both estimates in Lemma~\ref{L:interpolation-errors} are possible benchmarks for any approximation method based on the HHO space. Indeed, if $u \in \SobH{\Domain}$ solves \eqref{Poisson} and $\hU \in \HHOspace{}$ is the approximation resulting from a given HHO method, one may ask whether $\hU$ fulfills the same error bounds as $\Interp u$, possibly up to more pessimistic constants. This would guarantee that the method under examination reproduces the approximation properties of the underlying space. 

Unfortunately, the $H^1$- and the $L^2$-norm errors of the HHO method \eqref{HHO-method-standard} cannot be bounded like the corresponding interpolation errors in Lemma~\ref{L:interpolation-errors}. In fact, although it is certainly possible to relax the assumption that the load $f$ is in $\Leb{\Domain}$, as done in \cite{ErnGu:18}, the duality $\int_\Domain f \sigma_\Mesh$  in the right-hand side of \eqref{HHO-method-standard} cannot be continuously extended to general loads $f \in H^{-1} (\Domain)$ and arbitrary discrete test functions $\hsigma = (\sigma_\Mesh, \sigma_\Skeleton) \in \HHOspace{}$, because we possibly have $\sigma_\Mesh \notin H^1_0(\Domain)$. As a consequence, any error bound of \eqref{HHO-method-standard} must involve additional regularity beyond $f \in H^{-1}(\Domain)$ and $u \in \SobH{\Domain}$.  
%
Motivated by this observation, we aim at designing a variant of \eqref{HHO-method-standard} with improved approximation properties.


\section{A quasi-optimal variant of the HHO method}
\label{S:quasi-optimal-HHO}

In this section we exploit the abstract framework of section~\ref{S:abstract-framework} to design a new HHO method, which is quasi-optimal for \eqref{Poisson} in the semi-norm involved in the left-hand side of \eqref{interpolation-errors-H1-norm}. According to Remark~\ref{R:smoothing}, this requires, in particular, the use of a smoother in the discretization of the load. Hence, we first point out a condition on the smoother that is sufficient for quasi-optimality. Then, we construct a smoother fulfilling such a condition and derive broken $H^1$- and $L^2$-norm error estimates.  

\subsection{The HHO method with smoothing}
\label{SS:HHO-with-smoothing}

Let $\Mesh$ be the simplicial mesh introduced in section~\ref{SS:discrete-problem} and recall that the space $\HHOspace{}$ consists of pairs $\hs=(s_\Mesh, s_\Skeleton)$, where the first component $s_\Mesh \in \Poly{\Degree}{\Mesh}$ is an element-wise polynomial on $\Mesh$, whereas the second component $s_\Sigma \in \Poly{\Degree}{\Skeleton}$ is a face-wise polynomial on $\Skeleton$. Since the abstract framework of section~\ref{S:abstract-framework} involves the sum of continuous and discrete spaces, it is formally convenient to identify any element $v \in \SobH{\Domain}$ with the pair $\hv:=(v, v_{|\Skeleton})$, where $v_{|\Skeleton}$ denotes the trace of $v$ on $\Skeleton$. In fact, the Poisson problem \eqref{Poisson} fits into the abstract elliptic problem \eqref{cont-prob-abstract} provided we rewrite it as follows:
\begin{equation}
\label{cont-prob-abstract-HHO}
\text{Given } \; \ell_{\HHO} \in \DualSp{\SpaceContHHO}, 
\; \text{ find } \; \hu \in \SpaceContHHO \;\text{ s.t. }\;
\forall \hw \in \SpaceContHHO \quad a_{\HHO}(\hu,\hw) = \left\langle \ell_{\HHO}, \hw\right\rangle_{\DualSp{\SpaceContHHO} \times \SpaceContHHO},
\end{equation}
with the space
\begin{subequations}
\label{cont-prob-concrete}
\begin{equation}
\label{cont-prob-concrete-space}
\SpaceContHHO
:=
\{ \hv=(v_\Mesh, v_\Skeleton) \mid v_\Mesh \in \SobH{\Domain}, \;v_\Skeleton = (v_\Mesh)_{|\Skeleton} \} 
\end{equation}
and the forms
\begin{equation}
\label{cont-prob-concrete-form}
a_\HHO(\hv,\hw)
:=
\int_\Domain \Grad v_\Mesh \cdot \Grad w_\Mesh 
\end{equation}
\begin{equation}
\label{cont-prob-concrete-load}
\left\langle \ell_{\HHO}, \hw \right\rangle_{\DualSp{\SpaceContHHO} \times \SpaceContHHO} 
:= 
\left\langle f, w_\Mesh \right\rangle_{\SobHD{\Domain} \times \SobH{\Domain}}
\end{equation}	
\end{subequations}
where $\hv=(v_\Mesh, v_\Skeleton)$ and  $\hw=(w_\Mesh, w_\Skeleton)$ are in $\SpaceContHHO$ and $f \in \SobHD{\Domain}$ is the load in \eqref{Poisson}. This way of looking at the model problem \eqref{cont-prob-abstract} is instrumental to the derivation of Proposition~\ref{P:quasi-optimality-HHO}, although it may appear a bit artificial at first glance. 
 
The intersection of $\SpaceContHHO$ and the HHO space can be characterized as follows:
\begin{equation*}
\label{intersection}
\SpaceContHHO \cap \HHOspace{} 
=
\{
\hv = (v_\Mesh, v_\Skeleton) \in \SpaceContHHO 
\mid
v_\Mesh \in \Poly{\Degree}{\Mesh}
\}.
\end{equation*}
In particular, any element $\hv = (v_\Mesh, v_\Skeleton) \in \SpaceContHHO \cap \HHOspace{}$ satisfies 
\begin{equation}
\label{intersection-properties}
\Interp v_\Mesh = \hv, \quad \EllPro v_\Mesh = v_\Mesh, \quad
\PotRec \hv = v_\Mesh, \quad
\PotRecSt \hv = v_\Mesh, \quad
\Stab(\hv, \cdot) = 0.
\end{equation}

Proceeding as in section~\ref{S:abstract-framework}, we look for a symmetric bilinear form $\aextHHO$ on $\SpaceContHHO + \HHOspace{}$ such that $\aext_{\HHO|\SpaceContHHO} = a_\HHO$ and $\aext_{\HHO|\HHOspace{}} = \HHObform$. In other words, we require that $\aextHHO$ is a common extension of $a_\HHO$ and $\HHObform$.
It is readily seen that we must have 
\begin{equation}
\label{aext-HHO}
\aextHHO(\hv+\hs, \hw+\hsigma)
:=
\int_\Domain \GradM (v_\Mesh + \PotRec \hs) 
\cdot \GradM (w_\Mesh + \PotRec \hsigma)
+
\Stab(\hs,\hsigma),
\end{equation}
for all $\hv,\hw \in \SpaceContHHO$ and $\hs,\hsigma \in \HHOspace{}$. To check that $\aextHHO$ is indeed well-defined, assume that $\hv+\hs = \hv' + \hs'$ for some $\hv' \in \SpaceContHHO$ and $\hs' \in \HHOspace{}$. Then, we have $\hv-\hv' = \hs' - \hs \in \SpaceContHHO \cap \HHOspace{}$, so that \eqref{intersection-properties} implies $v_\Mesh - v'_\Mesh = \PotRec(\hv - \hv') = \PotRec(\hs'  - \hs)$ and $\Stab(\hs - \hs',\cdot) = 0$. Rearranging terms, we infer that $\aextHHO(\hv+\hs, \cdot) = \aextHHO(\hv' + \hs', \cdot)$. This observation and the symmetry of $\aextHHO$ confirm our claim. 

Let $\HHOsmt: \HHOspace{} \to \SobH{\Domain}$ be a linear operator. Motivated by Remark~\ref{R:smoothing}, we consider the following variant of the HHO method \eqref{HHO-method-standard}:
\begin{equation}
\label{HHO-method-modified}
\begin{split}
\text{Given } \; f \in \SobHD{\Domain}, 
\; \text{find } \; \hU \in \HHOspace{} 
\;\text{ s.t. } \quad\\
\forall \hsigma \in \HHOspace{}
\;\:
\HHObform(\hU, \hsigma)
=
\left\langle f, \HHOsmt \hsigma \right\rangle_{\SobHD{\Domain} \times \SobH{\Domain}}.
\end{split}
\end{equation}
Note, in particular, that here the right-hand side is defined for all $f \in \SobHD{\Domain}$. 

The new HHO method \eqref{HHO-method-modified} fits into the abstract discrete problem \eqref{disc-prob-abstract} with
\begin{equation}
	\label{disc-prob-concrete}
	\SpaceDisc = \HHOspace{}
	\qquad \qquad
	\aext = \aextHHO
	\qquad \qquad
	\Smt \hsigma = \hHHOsmt\hsigma := ( \HHOsmt \hsigma, (\HHOsmt \hsigma)_{|\Skeleton} )
\end{equation}
so that $\hHHOsmt : \HHOspace{} \rightarrow \SpaceContHHO$. Since $\HHOspace{} \nsubseteq \SpaceContHHO$, this is a nonconforming method. 

\subsection{Quasi-optimality}
\label{SS:quasi-optimality}

The extended energy semi-norm induced by the extended bilinear form $\aextHHO$ is
\begin{equation*}
\label{extended-seminorm}
\normHHO{\hv+\hs} := \sqrt{\aextHHO(\hv+\hs, \hv+\hs)}
\end{equation*}
with $\hv \in \SpaceContHHO$  and $\hs \in \HHOspace{}$.  
This is the unique common extension of the energy norm induced by $a_\HHO$ and the discrete norm induced by $\HHObform$.
We now aim at determining the properties of $\Smt_\HHO$ that are relevant for the quasi-optimality of \eqref{HHO-method-modified} in the semi-norm $\normHHO{\cdot}$. For this purpose, an important preliminary observation is that the setting proposed above does not fit into the abstract framework of section~\ref{S:abstract-framework}. In fact, the extended bilinear form $\aextHHO$ is only positive semi-definite on the sum $\SpaceContHHO + \HHOspace{}$, although its restrictions to $\SpaceContHHO$ and $\HHOspace{}$ are indeed positive definite. The following result makes our claim more precise. 

\begin{lemma}[Kernel of $\normHHO{\cdot}$]
	\label{L:degeneracy}
	We have $\normHHO{\hv-\hs} = 0$ for $\hv=(v_\Mesh, v_\Skeleton) \in \SpaceContHHO$ and $\hs \in \HHOspace{}$ if and only if $v_\Mesh \in \Poly{\Degree+1}{\Mesh}$ and $\hs = \Interp v_\Mesh$. 
\end{lemma}
\begin{proof}
Assume first that $v_\Mesh \in \Poly{\Degree+1}{\Mesh}$ and $\hs = \Interp v_\Mesh$. Owing to \eqref{RI=E}, we have $\PotRec \hs = v_\Mesh=\PotRecSt \hs$. The first identity implies that $\normLeb{\GradM(v_\Mesh - \PotRec \hs)}{\Domain} = 0$. The second one and the fact that $v_\Mesh \in \SobH{\Domain}$ reveal that $\Stab(\hs,\hs) = 0$. We conclude that $\normHHO{\hv-\hs} = 0$.

Conversely, assume that $\hv \in \SpaceContHHO$ and $\hs \in \HHOspace{}$ are such that $\normHHO{\hv-\hs} = 0$. This implies, in particular, that $\GradM(v_\Mesh - \PotRec \hs) = 0$. Therefore, we have $v_\Mesh \in \Poly{\Degree+1}{\Mesh}$. Hence, arguing as above, we infer the identity $\normHHO{\hv-\Interp v_\Mesh} = 0$, and the triangle inequality yields $\normHHO{\hs-\Interp v_\Mesh} = 0$. Since $\normHHO{\cdot}$ coincides with the norm induced by $\HHObform$ on $\HHOspace{}$, we conclude that $\hs = \Interp v_\Mesh$, owing to the coercivity of $\HHObform$ stated in Lemma~\ref{L:discrete-stability}. 
\end{proof}

\begin{remark}[Degeneracy of $\aextHHO$]
\label{R:degeneracy}
Let $\hv \in \SpaceContHHO$ and $\hs \in \HHOspace{}$ be such that $\normHHO{\hv-\hs}= 0$. The 'only if' part of Lemma~\ref{L:degeneracy} entails that we have two possibilities. If the cell component $v_\Mesh$ of $\hv$ is in $\Poly{\Degree}{\Mesh}$, then we have $\hv = \Interp v_\Mesh = \hs$. If, instead, $v_\Mesh \in \Poly{\Degree+1}{\Mesh} \setminus \Poly{\Degree}{\Mesh}$, then we have $\hv \neq \hs$, because $\hv$ is not in $\HHOspace{}$. On the one hand, this confirms that $\aextHHO$ is not positive definite on $\SpaceContHHO + \HHOspace{}$. On the other hand, we see that the difference $\hv-\hs$ is a nonzero element in the kernel of $\normHHO{\cdot}$ if and only if $\hv$ and $\hs$ are different pairs but $v_\Mesh$ coincides with the reconstruction of $\hs$. This originates from the fact that $\HHOspace{}$ is mapped by $\PotRec$ into a different space, which is 'one degree higher'.
\end{remark}

One possibility to deal with the degeneracy of $\aextHHO$ would be to take the quotient of $\SpaceContHHO + \HHOspace{}$ over the kernel of $\normHHO{\cdot}$. Another, actually equivalent, option is to replace the intersection $\SpaceContHHO \cap \HHOspace{}$ in the consistency condition \eqref{full-consistency} with the space of all pairs in $\SpaceContHHO$ whose distance to $\HHOspace{}$ vanishes in the semi-norm $\normHHO{\cdot}$, i.e. 
\begin{equation}
\label{kernel}
\hZ  :=
\{ \hz \in \SpaceContHHO \mid \inf_{\hs \in \HHOspace{}}\normHHO{\hz-\hs} = 0  \}
=
\{ \hz  \in \SpaceContHHO \mid \normHHO{\hz-\Interp z_\Mesh} = 0  \}.
\end{equation}
Notice that the second equality follows from Lemma~\ref{L:degeneracy}. 

Quasi-optimality in the semi-norm $\normHHO{\cdot}$ prescribes that the error of \eqref{HHO-method-modified} vanishes whenever the corresponding solution of \eqref{cont-prob-abstract-HHO} belongs to $\hZ$. This is a more restrictive consistency condition than \eqref{full-consistency}, because $\SpaceContHHO \cap \HHOspace{}$ is a strict subspace of $\hZ$.    

\begin{lemma}[Consistency conditions]
\label{L:algebraic-consistency}
Assume that $\hu \in \SpaceContHHO$ solves the problem \eqref{cont-prob-abstract-HHO} and denote by $\hU \in \HHOspace{}$ the solution of \eqref{HHO-method-modified}. The following conditions are equivalent:	
\begin{subequations}
\label{algebraic-consistency}
\begin{align}
\label{algebraic-consistecy-seminorm}
&\hu \in \hZ  
\quad \Longrightarrow \quad 
\normHHO{\hu-\hU} = 0\\
\label{algebraic-consistency-interp}
&\hu \in \hZ
\quad \Longrightarrow \quad
\hU = \Interp u_\Mesh\\
\label{algebraic-consistency-bform}
&\hu \in \hZ
\quad \Longrightarrow \quad
\left( \;\forall \hsigma \in \HHOspace{}, \quad 
\aextHHO(\hu, \hsigma-\hHHOsmt\hsigma) = 0 \;\right) 
\end{align}
\end{subequations}
and are necessary for quasi-optimality in the semi-norm $\normHHO{\cdot}$.
\end{lemma}
\begin{proof}
Let $\hu \in \hZ$. The second identity in \eqref{kernel} entails that $\aextHHO(\hu-\Interp u_\Mesh, \cdot) = 0$. Comparing also \eqref{cont-prob-abstract-HHO} with \eqref{HHO-method-modified} and recalling that $\aextHHO$ extends $\HHObform$, we see that
\begin{align*}
& \HHObform(\Interp u_\Mesh,\hsigma) = \aextHHO(\Interp u_\Mesh,\hsigma) = \aextHHO(\hu,\hsigma), \\
& \HHObform(\hU,\hsigma) = \left\langle f, \HHOsmt \hsigma \right\rangle_{\SobHD{\Domain} \times \SobH{\Domain}} = \aextHHO(\hu,\HHOsmt \hsigma).
\end{align*}
These identities reveal that the following is an equivalent reformulation of \eqref{algebraic-consistency-bform}:
\begin{equation*}
\hu \in \hZ
\quad \Longrightarrow \quad
\left( \;\forall \hsigma \in \HHOspace{}, \quad 
\HHObform(\Interp u_\Mesh - \hU, \hsigma) = 0 \;\right).
\end{equation*}
Thus, we infer that \eqref{algebraic-consistency-interp} $\Longleftrightarrow$ \eqref{algebraic-consistency-bform} owing to the nondegeneracy of $\HHObform$, whereas the equivalence \eqref{algebraic-consistecy-seminorm} $\Longleftrightarrow$ \eqref{algebraic-consistency-interp} is a consequence of Lemma~\ref{L:degeneracy}. Finally, the fact that \eqref{algebraic-consistecy-seminorm} is necessary for quasi-optimality in the semi-norm $\normHHO{\cdot}$ readily follows from \eqref{quasi-optimality} and the definition of $\hZ$.
\end{proof} 

Recall from \eqref{disc-prob-concrete} that the smoother $\hHHOsmt: \HHOspace{} \to \SpaceContHHO$ is obtained by means of the linear operator $\HHOsmt: \HHOspace{} \to \SobH{\Domain}$ which is used in the right-hand side of \eqref{HHO-method-modified}. Owing to the definition of $\aextHHO$, condition \eqref{algebraic-consistency-bform} can be further rewritten as follows:
\begin{equation*}
\forall \hu \in \hZ, \; \hsigma \in \HHOspace{}
\qquad
\int_\Domain \Grad u_\Mesh \cdot \GradM \PotRec \hsigma
= 
\int_\Domain \Grad u_\Mesh \cdot \Grad \HHOsmt \hsigma.
\end{equation*} 
Similar conditions can be found in \cite[Section~3.3]{Veeser.Zanotti:17p2} and \cite[Section~3.2]{Veeser.Zanotti:17p3} and are enforced there by means of moment-preserving smoothers, i.e., smoothers preserving certain moments on the simplices and on the interfaces of $\Mesh$. The integration by parts formula and the definition of the reconstruction allow us to apply the same technique also in this context. 

In what follows, we adopt the convention $\mathbb{P}_{-1} = \{0\}$. 
\begin{lemma}[Consistency via moment-preserving smoothers]
\label{L:moment-preservation}
Let $\hsigma = (\sigma_\Mesh, \sigma_\Skeleton)$ be any pair in $\HHOspace{}$ and assume that the operator $\HHOsmt: \HHOspace{} \to \SobH{\Domain}$ is such that
\begin{equation}
\label{moment-preservation}
\int_K q (\HHOsmt \hsigma)
=
\int_K q \sigma_\Mesh 
\qquad \text{ and } \qquad
\int_F r (\HHOsmt \hsigma)
=
\int_F r \sigma_\Skeleton
\end{equation}
for all $K \in \Mesh$, $q \in \Poly{\Degree-1}{K}$ and for all $F \in \FacesMint$, $r \in \Poly{\Degree}{F}$. Let $\hHHOsmt$ be defined as in \eqref{disc-prob-concrete}. Then, we have
\begin{equation}
\label{algebraic-consistency-suffcond}
\aextHHO(\hs, \hsigma - \hHHOsmt \hsigma) = \Stab(\hs,\hsigma),
\end{equation}
for all $\hs \in \HHOspace{}$. Moreover, \eqref{algebraic-consistency} holds true.
\end{lemma}
\begin{proof}
Let $\hsigma = (\sigma_\Mesh, \sigma_\Skeleton) \in \HHOspace{}$ be given. The definitions of $\PotRec$ and $\aextHHO$ in \eqref{gradient-reconstruction} and \eqref{aext-HHO}, respectively, yield
\begin{equation*}
\aextHHO(\hs,\hsigma) - \Stab(\hs,\hsigma)
=
\sum \limits_{K \in \Mesh}
\left( -\int_K (\Lapl \PotRec \hs) \sigma_\Mesh
+ \sum_{F \in \Faces{K}} \int_{F} (\Grad \PotRec \hs \cdot \Normal_K) \sigma_\Skeleton \right) 
\end{equation*}
for all $\hs \in \HHOspace{}$. Indeed, the fact that $\PotRec \hs \in \Poly{\Degree+1}{\Mesh}$ ensures that $\PotRec \hs$ is an admissible test function in \eqref{gradient-reconstruction-grad}. Moreover, since $\PotRec \hs$ is element-wise smooth, we can exploit once more the definition of $\aextHHO$ and integrate by parts element-wise. We obtain 
\begin{equation*}
\aextHHO(\hs, \hHHOsmt \hsigma)
=
\sum \limits_{K \in \Mesh} 
\left( -\int_K (\Lapl \PotRec \hs ) \HHOsmt \hsigma
+ \sum_{F \in \Faces{K}}\int_{F} (\Grad \PotRec \hs \cdot \Normal_K) \HHOsmt \hsigma \right),
\end{equation*}
for all $\hs \in \HHOspace{}$. Comparing this identity with the previous one and invoking assumption \eqref{moment-preservation}, we infer that \eqref{algebraic-consistency-suffcond} holds true. 

Next, let $\hu=(u_\Mesh, u_\Skeleton) \in \hZ$. The combination of \eqref{RI=E} with Lemma~\ref{L:degeneracy} reveal that $\PotRec \Interp u_\Mesh = u_\Mesh$ as well as $\Stab(\Interp u_\Mesh, \cdot) = 0$. Setting $\hs = \Interp u_\Mesh$ in \eqref{algebraic-consistency-suffcond}, we infer that
\begin{equation*}
\aextHHO(\hu, \hHHOsmt \hsigma)
=
\aextHHO(\Interp u_\Mesh, \hHHOsmt \hsigma) 
=
\aextHHO(\Interp u_\Mesh, \hsigma)
=
\aextHHO(\hu, \hsigma),
\end{equation*} 
for all $\hsigma \in \HHOspace{}$, showing that \eqref{algebraic-consistency} holds true.
\end{proof}

The importance of the identity \eqref{algebraic-consistency-suffcond} in Lemma~\ref{L:moment-preservation} goes beyond the fact that it is instrumental to check the validity of the consistency condition \eqref{algebraic-consistency}. Roughly speaking, it can be exploited also to bound the consistency error of \eqref{HHO-method-modified} in the so-called second Strang lemma \cite{Berger.Scott.Strang:72}. This is the key ingredient not only to prove the quasi-optimality of \eqref{HHO-method-modified} in the semi-norm $\normHHO{\cdot}$, but also to bound the corresponding quasi-optimality constant. 

\begin{proposition}[Quasi-optimality]
\label{P:quasi-optimality-HHO}
Assume that $\hu \in \SpaceContHHO$ solves the problem \eqref{cont-prob-abstract-HHO} and denote by $\hU \in \HHOspace{}$ the solution of \eqref{HHO-method-modified}. If the operator $\HHOsmt$ satisfies \eqref{moment-preservation}, then we have
\begin{equation}
\label{quasi-optimality-HHO}
\normHHO{\hu-\hU}
\leq
\sqrt{1 + C_{\HHO}^2}
\inf \limits_{\hs \in \HHOspace{}}
\normHHO{\hu-\hs},
\end{equation}
where $C_\HHO$ is the smallest constant such that
\begin{equation}
\label{consistency-constant}
\forall \hsigma \in \HHOspace{} \qquad
\normLeb{\GradM(\PotRec\hsigma - \HHOsmt \hsigma)}{\Domain} 
\leq C_\HHO \normHHO{\hsigma}.
\end{equation} 
\end{proposition}
\begin{proof}
We adapt the approach devised in \cite[section~3]{Veeser.Zanotti:17p3} to our context. Denote by $\Ritz: \SpaceContHHO \to \HHOspace{}$ the $\aextHHO$-orthogonal projection onto $\HHOspace{}$, i.e.,
\begin{equation}
\label{Ritz-projection}
\forall \hsigma \in \HHOspace{} \qquad
\aextHHO(\Ritz \hv, \hsigma) = \aextHHO(\hv,\hsigma)
\end{equation}
for all $\hv \in \SpaceContHHO$. Notice that this problem is uniquely solvable (because $\aextHHO$ restricted to $\HHOspace{}$ is positive-definite) and that $\Ritz \hv$ is the best approximation of $\hv$ in $\HHOspace{}$ with respect to the semi-norm $\normHHO{\cdot}$. The $\aextHHO$-orthogonality of $\Ritz$ implies that
\begin{equation}
\label{quasi-optimality-proof1}
\normHHO{\hu-\hU}^2 =
\normHHO{\hu-\Ritz \hu}^2 + \normHHO{\hU - \Ritz \hu}^2.
\end{equation}
Since $\aextHHO$ is a scalar product on $\HHOspace{}$, we have
\begin{equation}
\label{quasi-optimality-proof2}
\normHHO{\hU-\Ritz \hu}
=
\sup \limits_{\hsigma \in \HHOspace{}}
\dfrac{\aextHHO(\hU-\Ritz \hu, \hsigma)}{\normHHO{\hsigma}}.
\end{equation}
Let $\hsigma \in \HHOspace{}$ be arbitrary and recall that the restriction of $\aextHHO$ to $\HHOspace{}$ coincides with $\HHObform$. A comparison of problems \eqref{cont-prob-abstract} and \eqref{HHO-method-modified} reveals that
\begin{equation*}
\aextHHO(\hU-\Ritz \hu, \hsigma)  =
\aextHHO(\hu, \HHOsmt \hsigma) - \aextHHO(\Ritz \hu, \hsigma) =
\aextHHO(\hu-\Ritz \hu, \HHOsmt \hsigma) - \Stab(\Ritz \hu , \hsigma),
\end{equation*}
where the second identity follows from Lemma~\ref{L:moment-preservation}. Rearranging terms in \eqref{Ritz-projection} and recalling the expression of $\aextHHO$ in \eqref{aext-HHO}, we infer that
\begin{equation*}
\Stab(\Ritz \hu, \hsigma) =
\int_\Domain \GradM(u_\Mesh - \PotRec \Ritz \hu) \cdot \GradM \PotRec \hsigma
\end{equation*}
where $u_\Mesh$ is the cell component of $\hu$. If we insert this identity into the previous one, we infer that
\begin{equation*}
\aextHHO(\hU-\Ritz \hu, \hsigma)  =
\int_\Domain \GradM(u_\Mesh - \PotRec \Ritz \hu) \cdot \GradM( \HHOsmt\hsigma - \PotRec \hsigma).
\end{equation*}
Comparing with \eqref{quasi-optimality-proof2} and recalling the definition of $C_\HHO$ in \eqref{consistency-constant}, we finally obtain that
\begin{equation*}
\normHHO{\hU-\Ritz \hu} \leq 
C_\HHO \normHHO{\hu-\Ritz \hu}.
\end{equation*}
We conclude by inserting this inequality into \eqref{quasi-optimality-proof1}. 
\end{proof}


\subsection{Moment-preserving smoothers}
\label{SS:smoother}

Motivated by Proposition~\ref{P:quasi-optimality-HHO}, we now aim at constructing a concrete smoother which fulfills \eqref{moment-preservation} and such that the constant $C_\HHO$ in \eqref{consistency-constant} is $\leq C_{\Shape, \Degree}$. To make sure that our construction is of practical interest, we also require that the smoother is computationally feasible in the sense of Remark~\ref{R:feasible-smoothing}. As before, we denote by $\Dim \in \{2,3\}$ the space dimension and use the convention $\mathbb{P}_{-1} = \{0\}$. Our construction is inspired by the one in \cite[Section~3.3]{Veeser.Zanotti:17p2}. 

For all $K \in \Mesh$, we denote by $\Phi_K \in \SobH{\Domain}$ the \textit{element bubble} determined by the conditions $(i)$ $\Phi_K \equiv 0$ in $\overline{\Domain} \setminus K$, $(ii)$ $(\Phi_K)_{|K} \in \Poly{\Dim + 1}{K}$ and $(iii)$ $\Phi_K(m_K) = 1$ at the barycenter $m_K$ of $K$. We introduce a local linear operator $\BubbOper{K}: \Leb{\Domain} \to \Poly{\Degree-1}{K}$ by setting
\begin{equation}
	\label{smoother-cell}
	\forall q \in \Poly{\Degree-1}{K}
	\qquad
	\int_K q (\BubbOper{K} v_\Mesh) \Phi_K
	=
	\int_K q v_\Mesh,
\end{equation}      
for all $v_\Mesh\in \Leb{\Domain}$. Then, the global operator $\BubbOper{\Mesh}: \Leb{\Domain} \to \SobH{\Domain}$ is defined such that
\begin{equation}
	\label{smoother-mesh}
	\BubbOper{\Mesh} v_\Mesh
	:=
	\sum \limits_{K \in \Mesh} (\BubbOper{K} v_\Mesh) \Phi_K.
\end{equation}
Since $(\BubbOper{\Mesh} v_\Mesh)_{|K} = (\BubbOper{K} v_\Mesh)\Phi_K$, the operator $\BubbOper{\Mesh}$
preserves all the moments of $v_\Mesh$ of degree $\leq \Degree-1$ in each simplex of $\Mesh$, as a consequence of \eqref{smoother-cell}.

Next, let $F \in \FacesMint$ be an interface and let $K_1, K_2 \in \Mesh$ be such that $F = K_1 \cap K_2$. Setting $\omega_F := K_1 \cup K_2$, we denote by $\Phi_F \in \SobH{\Domain}$ the \textit{face bubble} determined by the conditions $(i)$ $\Phi_F \equiv 0$ in $\overline{\Domain} \setminus \omega_F$, $(ii)$ $(\Phi_{F})_{|K_j} \in \Poly{\Dim}{K_j}$ for $j=1,2$ and $(iii)$ $\Phi_F(m_F) = 1$ at the barycenter $m_F$ of $F$. We introduce a local linear operator $\BubbOper{F}: \Leb{\Skeleton} \to \Poly{\Degree}{F}$ setting
\begin{equation}
	\label{smoother-face}
	\forall r \in \Poly{\Degree}{F}
	\qquad
	\int_F r (\BubbOper{F} v_\Skeleton) \Phi_F
	=
	\int_F r v_\Skeleton
\end{equation}   
for all $v_\Skeleton\in \Leb{\Skeleton}$.
%
For $\Degree = 0$, it is straightforward to extend $\BubbOper{F} v_\Skeleton$ from $\Poly{0}{F}$ to $\Sob{\Domain} \cap \Poly{0}{\Mesh}$. For $\Degree \geq 1$, let $\Lagr{\Degree}{\Mesh}$ collect the Lagrange nodes of degree $\Degree$ of $\Mesh$. For each $z \in \Lagr{\Degree}{\Mesh}$, let $\Phi_z$ be the Lagrange basis function of $\Sob{\Domain} \cap \Poly{\Degree}{\Mesh}$ associated with the evaluation at $z$, that is $\Phi_z (z') = \delta_{zz'}$ for all $z' \in \Lagr{\Degree}{\Mesh}$. Since $\Mesh$ is a matching simplicial mesh, the set $\{ (\Phi_z)_{|F} \mid z \in \Lagr{\Degree}{\Mesh} \cap F \}$ is the Lagrange basis of $\Poly{\Degree}{F}$. Therefore, we have $\BubbOper{F} v_\Skeleton = \sum_{z \in \Lagr{\Degree}{\Mesh} \cap F} (\BubbOper{F} v_\Skeleton)(z) \Phi_z$ in $F$. Motivated by this identity, we define the global operator $\BubbOper{\Skeleton}: \Leb{\Skeleton} \to \SobH{\Domain}$ such that
\begin{equation}
\label{smoother-skeleton}
\BubbOper{\Skeleton} v_\Skeleton
:=
\begin{cases}
\sum \limits_{F \in \FacesMint}
(\BubbOper{F} v_\Skeleton)\Phi_F, & \Degree =0,\\[3mm]
\sum \limits_{F \in \FacesMint}
\sum\limits_{z \in \Lagr{\Degree}{\Mesh} \cap F} (\BubbOper{F} v_\Skeleton)(z) \Phi_z \Phi_F, & \Degree \geq 1.
\end{cases}
\end{equation}
Since $(\BubbOper{\Skeleton} v_\Skeleton)_{|F} = (\BubbOper{F} v_\Skeleton)\Phi_F$ for all $F \in \FacesMint$ and all $\Degree\ge0$, the identity \eqref{smoother-face} implies that the operator $\BubbOper{\Skeleton}$ preserves all the moments of $v_\Skeleton$ of degree $\leq \Degree$ on each interface of $\Mesh$. 

A proper combination of $\BubbOper{\Mesh}$ and $\BubbOper{\Skeleton}$ provides an operator $\BubbOper{}$ which preserves all moments prescribed in \eqref{moment-preservation}.
\begin{proposition}[Bubble smoother]
\label{P:bubble-smoother} The operator $\BubbOper{}: \Leb{\Domain} \times \Leb{\Skeleton} \to \SobH{\Domain}$ defined for all $\hv=(v_\Mesh,v_\Skeleton) \in \Leb{\Domain} \times \Leb{\Skeleton}$ such that
\begin{equation}
\label{bubble-smoother}
\BubbOper{} \hv 
:=
\BubbOper{\Skeleton} v_\Skeleton 
+
\BubbOper{\Mesh} (v_\Mesh - \BubbOper{\Skeleton} v_\Skeleton)
\end{equation}
fulfills \eqref{moment-preservation} and satisfies, for all $K \in \Mesh$, the following estimate:
\begin{equation}
\label{bubble-smoother-stability}
\normLeb{\Grad \BubbOper{} \hv}{K}
\leq C_{\Shape,\Degree} \left( 
h_K^{-1} \normLeb{v_\Mesh}{K} +
\sum_{F \in \Faces{K}} h_F^{-\frac{1}{2}} \normLeb{v_\Skeleton}{F}\right).
\end{equation}
\end{proposition}  

\begin{proof}
Let $\hv = (v_\Mesh, v_\Skeleton) \in \Leb{\Domain} \times \Leb{\Skeleton}$. Owing to the definition of $\BubbOper{\Mesh}$ and \eqref{smoother-cell}, we have
\begin{equation}
\label{bubble-smoother-cell}
\int_K q (\BubbOper{} \hv) 
=
\int_K q (\BubbOper{\Skeleton} v_\Skeleton)
+
\int_K q (\BubbOper{K}( v_\Mesh - \BubbOper{\Skeleton} v_\Skeleton)) \Phi_K
=
\int_K q v_\Mesh,
\end{equation}
for all $K \in \Mesh$ and $q \in \Poly{\Degree-1}{K}$. Moreover, since $\BubbOper{\Mesh} (v_\Mesh - \BubbOper{\Skeleton} v_\Skeleton)$ vanishes on the skeleton of $\Mesh$, the definition of $\BubbOper{\Skeleton}$ and \eqref{smoother-face} reveal that
\begin{equation}
\label{bubble-smoother-faces}
\int_F r (\BubbOper{} \hv) 
=
\int_F r (\BubbOper{\Skeleton} v_\Skeleton)
=
\int_F r (B_F v_\Skeleton) \Phi_F
=
\int_F r v_\Skeleton,
\end{equation}
for all $F \in \FacesMint$ and $r \in \Poly{\Degree}{F}$. The above identities confirm that $\BubbOper{}$ fulfills \eqref{moment-preservation}. 

To verify the claimed $H^1$-norm estimate \eqref{bubble-smoother-stability}, fix $K \in \Mesh$ and $\hv = (v_\Mesh, v_\Skeleton)\in \Leb{\Domain} \times \Leb{\Skeleton}$. The definition of $\BubbOper{\Mesh}$ and $\Phi_K \leq 1$ yield
\begin{equation*}
\normLeb{\BubbOper{\Mesh} v_\Mesh}{K}^2
\leq
\int_K (\BubbOper{K} v_\Mesh) v_\Mesh
\leq
\normLeb{\BubbOper{K} v_\Mesh}{K} \normLeb{v_\Mesh}{K}.
\end{equation*}
Hence, we obtain $\normLeb{\BubbOper{\Mesh} v_\Mesh}{K} \Cleq \normLeb{v_\Mesh}{K}$ by a standard argument with bubble functions, see \cite{Verfurth:13}. Next, for $\Degree \geq 1$, the boundedness of the extension employed in \eqref{smoother-skeleton} and a scaling argument imply that
\begin{equation*}
\normLeb{\BubbOper{\Skeleton} v_\Skeleton}{K}^2
\Cleq
\sum_{F \in \Faces{K}}
\sum_{z \in \Lagr{p}{\Mesh} \cap F} 
\normSemi{\BubbOper{F} v_\Skeleton(z)\Phi_F(z) }^2 
\Cleq
\sum_{F \in \Faces{K}} h_F
\normLeb{(\BubbOper{F} v_\Skeleton)\Phi_F }{F}^2.
\end{equation*}
Apart of the intermediate step, the same estimate holds also for $\Degree = 0$. Then, for all $F \in \Faces{K}$, we argue as before, noticing $\Phi_F \leq 1$, to infer that
\begin{equation*}
\normLeb{(\BubbOper{F} v_\Skeleton)\Phi_F}{F}^2
\leq
\int_F (\BubbOper{F} v_\Skeleton)^2 \Phi_F
=
\int_F (\BubbOper{F} v_\Skeleton) v_\Skeleton
\leq
\normLeb{\BubbOper{F} v_\Skeleton}{F} \normLeb{v_\Skeleton}{F}.
\end{equation*}
This entails $\normLeb{\BubbOper{\Skeleton}v_\Skeleton}{K} \Cleq \sum_{F \in \Faces{K}}h_F^{1/2} \normLeb{v_\Sigma}{F}$, by a standard argument with bubble functions, see \cite{Verfurth:13}. We conclude combining this bound and the previous one with the definition of $\BubbOper{}$ in \eqref{bubble-smoother} and with the inverse estimate $\normLeb{\Grad \BubbOper{} \hv}{K} \Cleq h_K^{-1} \normLeb{\BubbOper{} \hv}{K}$. 
\end{proof}

The bubble smoother $\BubbOper{}$ maps into a space of bubble functions, thus generating spurious oscillations. This simple observation and inequality \eqref{bubble-smoother-stability} suggest that the $H^1$-norm of $\BubbOper{}\hsigma$ cannot be uniformly bounded by the $\normSemi{\cdot}_{\aextHHO}$-norm of $\hsigma$, irrespective of the size of $\Mesh$, for arbitrary $\hsigma \in \HHOspace{}$. This claim can be verified arguing as in \cite[Remark~3.5]{Veeser.Zanotti:17p2}. Therefore, the bubble smoother $\BubbOper{}$ should not be used into the HHO method \eqref{HHO-method-modified}, although it preserves all the moments prescribed in \eqref{moment-preservation}. In fact, as mentioned in Remark~\ref{R:smoothing}, the quasi-optimality constant of a quasi-optimal method is bounded from below in terms of the operator norm of the employed smoother.

The inequality \eqref{bubble-smoother-stability} indicates that we may define a \textit{stabilized} version of $\BubbOper{}$ if we replace $\hv$ with $\hv - \hAvgOper{} \hv$ in \eqref{bubble-smoother}, provided $\hAvgOper{} \hv \in \SpaceContHHO$ is locally (at least) a first-order approximation of $\hv$. The operator $\hAvgOper{}$ can be defined, for instance, through some averaging technique, in the vein of \cite{Oswald:93,Karakashian.Pascal:03,Ern.Guermond:17}.    

To make things precise, denote by $\LagrInt{\Degree+1}{\Mesh}$ the interior Lagrange nodes of degree $\Degree+1$ of $\Mesh$ (i.e. the Lagrange nodes not lying on $\partial \Domain$). For each node $z \in \LagrInt{\Degree+1}{\Mesh}$, let $\Phi_z$ be the Lagrange basis function of $\SobH{\Domain} \cap \Poly{\Degree+1}{\Mesh}$ associated with the evaluation at $z$. We define $\AvgOper{}: \HHOspace{} \to \SobH{\Domain}$ such that
\begin{equation}
	\label{averaging}
	\AvgOper{} \hsigma
	:=
	\sum_{z \in \LagrInt{\Degree+1}{\Mesh}}
	\left( \dfrac{1}{\#\omega_z}
	\sum_{K \in \omega_z} (\PotRec \hsigma)_{|K}(z) 
	\right) 
	\Phi_z, 
\end{equation}
for all $\hsigma = (\sigma_\Mesh, \sigma_\Skeleton) \in \HHOspace{}$, where $\omega_z$ collects the simplices of $\Mesh$ to which $z$ belongs and $\#\omega_z$ denotes the cardinality of $\omega_z$. The next proposition confirms that we can use this operator to stabilize the bubble smoother $\BubbOper{}$. We discuss possible variants of $\AvgOper{}$ in Remark~\ref{R:averaging-variants} below. Notice that $\AvgOper{}$ should not be directly used in \eqref{HHO-method-modified}, because it may not preserve the moments prescribed in \eqref{moment-preservation}.

\begin{proposition}[Stabilized bubble smoother]
  \label{P:smoother-moment-preserving}
  Let $\BubbOper{}$ and $\AvgOper{}$ be defined as in \eqref{bubble-smoother} and \eqref{averaging}, respectively, and let $\hAvgOper{} : \HHOspace{} \to \SpaceContHHO$ be defined such that $\hAvgOper{} \hsigma:=(\AvgOper{} \hsigma,(\AvgOper{} \hsigma)_{|\Skeleton})$ for all $\hsigma = (\sigma_\Mesh, \sigma_\Skeleton) \in \HHOspace{}$. Then, the operator $\HHOsmt: \HHOspace{} \to \SobH{\Domain}$ such that
	\begin{equation}
		\label{smoother-moment-preserving}
		\HHOsmt \hsigma
		:=
		\AvgOper{} \hsigma + \BubbOper{}(\hsigma - \hAvgOper{} \hsigma)
	\end{equation}
	fulfills \eqref{moment-preservation} and is such that
	\begin{equation}
		\label{smoother-moment-preserving-stability}
		\normLeb{\GradM(\PotRec\hsigma - \HHOsmt \hsigma) }{\Domain}^2
		\leq C_{\Shape, \Degree}
		\sum_{K \in \Mesh} \sum_{F \in \Faces{K}}
		h_F^{-1} \normLeb{\sigma_\Skeleton - (\sigma_\Mesh)_{|K}}{F}^2.
	\end{equation}
\end{proposition} 

\begin{proof}
	According to \eqref{bubble-smoother-cell}, we have
	\begin{equation*}
		\int_K q\HHOsmt \hsigma
		=
		\int_K q(\AvgOper{} \hsigma - \BubbOper{} \hAvgOper{} \hsigma)
		+
		\int_K q \BubbOper{}\hsigma
		=
		\int_K q \sigma_\Mesh,
	\end{equation*}
	for all $K \in \Mesh$, $q \in \Poly{\Degree-1}{K}$ and $\hsigma \in \HHOspace{}$. The fact that $\HHOsmt$ preserves all the moments of degree $\leq \Degree$ on the interfaces of $\Mesh$ can be verified similarly, with the help of \eqref{bubble-smoother-faces}. This confirms that $\HHOsmt$ fulfills \eqref{moment-preservation}.
	
	Concerning the claimed stability, we first derive a local version of \eqref{smoother-moment-preserving-stability}. To this end, let $K \in \Mesh$ and $\hsigma \in \HHOspace{}$ be given. The triangle inequality readily implies that
	\begin{equation*}
	\normLeb{\Grad (\PotRec\hsigma - \HHOsmt\hsigma)}{K}
	\leq
	\normLeb{\Grad(\PotRec\hsigma - \AvgOper{} \hsigma)}{K}
	+
	\normLeb{\Grad \BubbOper{} (\hsigma - \hAvgOper{} \hsigma)}{K}.
	\end{equation*}
	We estimate the second summand in the right-hand side with the help of Proposition~\ref{P:bubble-smoother}, the discrete trace inequality \eqref{discrete-trace-inequality}, identity \eqref{gradient-reconstruction-avg} and the Poincar\'e-Steklov inequality \eqref{poincare-inequality}:
	\begin{equation*}
	\begin{split}
	&\normLeb{\Grad \BubbOper{} (\hsigma - \hAvgOper{} \hsigma)}{K} 
	\Cleq
	h_K^{-1} \normLeb{\sigma_\Mesh - \AvgOper{} \hsigma}{K}
	+
	\sum_{F \in \Faces{K}} h_F^{-\frac{1}{2}}
	\normLeb{\sigma_\Skeleton - \AvgOper{} \hsigma}{F}\\
	& \Cleq
	h_K^{-1} \normLeb{\PotRec\hsigma - \AvgOper{} \hsigma}{K}
	+
	\normLeb{\Grad(\sigma_\Mesh - \PotRec \hsigma)}{K}
	+
	\sum_{F \in \Faces{K}} h_F^{-\frac{1}{2}}
	\normLeb{\sigma_\Skeleton - (\sigma_\Mesh)_{|K}}{F}.
	\end{split}
	\end{equation*}
	We insert this bound into the previous one. An inverse estimate yields
	\begin{equation}
	\label{smoother-moment-preserving-proof1}
	\begin{split}
	\normLeb{\Grad (\PotRec\hsigma - \HHOsmt\hsigma)}{K}
	&\Cleq
	h_K^{-1} \normLeb{\PotRec\hsigma - \AvgOper{} \hsigma}{K}
	+
	\normLeb{\Grad(\sigma_\Mesh - \PotRec \hsigma)}{K}\\
	&\quad +
	\sum_{F \in \Faces{K}} h_F^{-\frac{1}{2}}
	\normLeb{\sigma_\Skeleton - (\sigma_\Mesh)_{|K}}{F}.
	\end{split}
	\end{equation} 
 	We estimate the first summand in the right-hand side by means of \cite[Lemma~4.3]{Ern.Guermond:17}. Invoking also \eqref{gradient-reconstruction-avg}, \eqref{discrete-trace-inequality} and \eqref{poincare-inequality}, we derive
	\begin{equation*}
	\label{smoother-moment-preserving-proof2}
	\begin{split}
	&h_K^{-1}\normLeb{\PotRec\hsigma - \AvgOper{} \hsigma}{K}
	\Cleq
	\sum_{F \cap K \neq \emptyset}
	h_F^{-\frac{1}{2}} \normLeb{\Jump{\PotRec\hsigma}}{F}\\
	&\;\; \Cleq
	\sum_{K' \cap K \neq \emptyset} \left(
	\normLeb{\Grad (\sigma_\Mesh - \PotRec \hsigma)}{K'}
	+
	\sum_{F' \in \Faces{K'}}   
	h_{F'}^{-\frac{1}{2}} \normLeb{\sigma_\Skeleton - (\sigma_\Mesh)_{|K'}}{F'}  \right)
	\end{split}
	\end{equation*}
	where $F$ and $K'$ vary in $\Faces{}$ and $\Mesh$, respectively, and $\Jump{\cdot}$ is the jump operator. Moreover, for all $K' \in \Mesh$, the identity \eqref{gradient-reconstruction-grad} and \eqref{discrete-trace-inequality} reveal
	\begin{equation*}
	\label{smoother-moment-preserving-proof3}
	\normLeb{\Grad(\sigma_\Mesh - \PotRec \hsigma)}{K'}
	\Cleq
	\sum_{F' \in \Faces{K'}} h_{F'}^{-\frac{1}{2}} 
	\normLeb{\sigma_\Skeleton - (\sigma_\Mesh)_{|K'}}{F'}.
	\end{equation*}
	We insert this bound and the previous one into \eqref{smoother-moment-preserving-proof1}. Squaring and summing over all $K \in \Mesh$, we infer that
	\begin{equation*}
	\normLeb{\GradM (\PotRec\hsigma - \HHOsmt\hsigma)}{\Domain}^2
	\Cleq
	\sum_{K \in \Mesh} \sum_{K' \cap K \neq \emptyset}
	\sum_{F' \in \Faces{K'}}
	h_{F'}^{-1} \normLeb{\sigma_\Skeleton - (\sigma_\Mesh)_{|K'}}{F'}^2,
	\end{equation*}
	where $K'$ varies in $\Mesh$. We conclude recalling that the number of simplices touching a given simplex is $\leq C_\Shape$.
\end{proof}

\begin{remark}[Variants of $\AvgOper{}$]
\label{R:averaging-variants}
Instead of taking the average of $\PotRec\hsigma$ at each node $z \in \LagrInt{\Degree+1}{\Mesh}$, it is possible to fix $K_z \in \Mesh$ with $z \in K_z$ and set 
\begin{equation}
\label{averaging-simplified}
\AvgOper{}' \hsigma := \sum_{z \in \LagrInt{\Degree+1}{\Mesh}} (\PotRec\hsigma)_{|K_z} \Phi_z
\end{equation}
in the vein of the Scott--Zhang interpolation \cite{Scott.Zhang:90}. This modification preserves the main properties of $\AvgOper{}$, whereas the operations needed compute $\AvgOper{}'$ are significantly reduced, see \cite[Lemma~3.3]{Veeser.Zanotti:17p3}. One may also replace the reconstruction $\PotRec\hsigma$ by the cell component $\sigma_\Mesh$ of $\hsigma$, both in \eqref{averaging} and \eqref{averaging-simplified}. Hence, for $\Degree \geq 1$, the sum can be restricted to the interior Lagrange nodes of degree $\Degree$ (and not $\Degree+1$). With this variant of $\AvgOper{}$ and $\AvgOper{}'$, the statement of Proposition~\ref{P:smoother-moment-preserving} remains unchanged. Yet, the proof of Lemma~\ref{L:supercloseness} below and the subsequent derivation of an $L^2$-norm error estimate appear to be problematic for $\Degree = 0$. 
\end{remark}  

\begin{remark}[Feasibility of $\HHOsmt$]
\label{R:computational-aspects}
Let $\HHOsmt$ be as in Proposition~\ref{P:smoother-moment-preserving}. A computationally convenient basis of $\HHOspace{}$ consists of functions $\hsigma_1,\dots,\hsigma_N$ that are supported either in one simplex or on one interface of $\Mesh$. The local estimates established in the proof of Proposition~\ref{P:smoother-moment-preserving} reveal that the support of $\HHOsmt\hsigma_i$, $i=1,\dots,N$, is a subset of $\bigcup \{K \in \Mesh \mid K \cap \mathrm{supp}(\hsigma_i) \neq \emptyset \}$. Hence, the number of simplices in the support of $\HHOsmt\hsigma_i$ is $\leq C_{\Shape}$. Moreover, the construction of $\HHOsmt \hsigma_i$ from $\hsigma_i$ requires at most $\mathrm{O}(1)$ operations. Therefore,  we can evaluate the duality $\left\langle f, \HHOsmt \hsigma_i \right\rangle_{\SobHD{\Domain} \times \SobH{\Domain}}$ with $\mathrm{O}(1)$ operations and the cost for solving \eqref{HHO-method-modified} is at most a constant factor times the cost for solving \eqref{HHO-method-standard}.
\end{remark}


\subsection{Error estimates}
\label{SS:error-estimates}

We now consider the HHO method \eqref{HHO-method-modified} with the smoother $\HHOsmt$ proposed in \eqref{smoother-moment-preserving} and derive broken $H^1$- and $L^2$-norm error estimates. The former readily follows from the abstract quasi-optimality stated in Proposition~\ref{P:quasi-optimality-HHO}, combined with the approximation properties of the HHO space and  Proposition~\ref{P:smoother-moment-preserving}.

\begin{theorem}[Broken $H^1$-norm error estimate]
	\label{T:H1-estimate}
	Let $u \in \SobH{\Domain}$ solve \eqref{Poisson} and denote by $\hU \in \HHOspace{}$ the solution of \eqref{HHO-method-modified} with $\HHOsmt$ as in Proposition~\ref{P:smoother-moment-preserving}. Then, the following holds true:
	\begin{equation}
		\label{H1-estimate-nonasymptotic}
		\normLeb{\GradM(u-\PotRec \hU)}{\Domain}^2
		+
		\Stab(\hU, \hU)
		\leq C_{\Shape,\Degree}
		\sum_{K \in \Mesh} \inf_{q \in \Poly{\Degree+1}{K}}
		\normLeb{\Grad(u-q)}{K}^2.
	\end{equation}
	Furthermore, if $u \in H^m(\Domain)$ with $m \in \{ 1,\dots, \Degree+2 \}$, we have
	\begin{equation}
		\label{H1-estimate-asymptotic}
		\normLeb{\GradM(u-\PotRec \hU)}{\Domain}^2
		+
		\Stab(\hU, \hU)
		\leq C_{\Shape,\Degree}
		\sum_{K \in \Mesh} h_K^{2(m-1)}
		\normSemi{u}_{H^{m}(K)}^2.
	\end{equation}
\end{theorem}

\begin{proof}
The combination of Propositions~\ref{P:quasi-optimality-HHO} and \ref{P:smoother-moment-preserving} ensures that the HHO method \eqref{HHO-method-modified} with $\HHOsmt$ as in \eqref{smoother-moment-preserving} is quasi-optimal in the semi-norm $\normHHO{\cdot}$. Recalling the definition of the semi-norm $\normHHO{\cdot}$, the quasi-optimal estimate \eqref{quasi-optimality-HHO} takes the form
\begin{equation*}
\normLeb{\GradM(u-\PotRec \hU)}{\Domain}^2
+
\Stab(\hU, \hU)
\leq (1 + C_{\HHO}^2)
\inf_{\hs \in \HHOspace{}}
\left( \normLeb{\GradM(u-\PotRec \hs)}{\Domain}^2
+ \Stab(\hs,\hs)\right).  
\end{equation*}
Lemma~\ref{L:discrete-stability} and Proposition~\ref{P:quasi-optimality-HHO} provide also an upper bound on $C_\HHO$. In fact, for all $\hsigma = (\sigma_\Mesh, \sigma_\Skeleton) \in \HHOspace{}$, we have
\begin{equation*}
\normLeb{\GradM(\PotRec \hsigma - \HHOsmt \hsigma)}{\Domain}
\Cleq
\norm{\hsigma}_{\HHOspace{}}
\Cleq \normHHO{\hsigma},
\end{equation*}
showing that $C_\HHO \leq C_{\Shape, \Degree}$. Thus, we infer that	
\begin{equation*}
	\normLeb{\GradM(u-\PotRec \hU)}{\Domain}^2
	+
	\Stab(\hU, \hU)
	\Cleq
	\inf_{\hs \in \HHOspace{}}
	\left( \normLeb{\GradM(u-\PotRec \hs)}{\Domain}^2
	+ \Stab(\hs,\hs)\right).  
\end{equation*}
We can now derive the first claimed estimate by taking $\hs=\Interp u$ and using inequality \eqref{interpolation-errors-H1-norm} in Lemma~\ref{L:interpolation-errors}. The second estimate easily follows from the first one using standard polynomial approximation properties in Sobolev spaces.
\end{proof}

According to Theorem~\ref{T:H1-estimate}, the HHO method \eqref{HHO-method-modified} with the smoother $\HHOsmt$ proposed in \eqref{smoother-moment-preserving} reproduces the approximation properties of the interpolant $\Interp$ (see \eqref{HHO-interpolant}) in the semi-norm $\normHHO{\cdot}$. In fact, similarly to the first estimate of Lemma~\ref{L:interpolation-errors}, the right-hand side of \eqref{H1-estimate-nonasymptotic} bounds the left-hand side also from below. Note also that only the minimal regularity $u \in \SobH{\Domain}$ is involved there and that \eqref{H1-estimate-asymptotic} exploits only element-wise regularity of $u$. 

Next, we recall from \cite[Theorem~10]{DiPietro.Ern.Lemaire:14} that an $L^2$-norm error estimate of the HHO method \eqref{HHO-method-standard} can be derived via the Aubin--Nitsche duality technique. We aim at establishing a counterpart of such a result in the present setting. This would confirm, in particular, that the use of a smoother does not generally rule out the possibility of establishing $L^2$-norm error estimates by duality.

As before, we denote by $u \in \SobH{\Domain}$ and $\hU = (U_\Mesh, U_\Sigma) \in \HHOspace{}$ the solutions of problems \eqref{Poisson} and \eqref{HHO-method-modified}, respectively, with $\HHOsmt$ as in \eqref{smoother-moment-preserving}. Proceeding as in \cite{DiPietro.Ern.Lemaire:14}, we let $\psi \in \SobH{\Domain}$ be the weak solution of 
\begin{equation}
	\label{duality}
	-\Lapl \psi = U_\Mesh - \LebPro_\Mesh u
	\;\; \text{ in } \Domain
	\qquad \text{ and }\qquad
	\psi = 0 
	\;\; \text{ on } \partial \Domain.
\end{equation}
By elliptic regularity \cite{Grisvard:11}, there are $\alpha \in (\frac12, 1]$ and a constant $c \geq 0$ such that $\psi \in H^{1+\alpha}(\Domain)$ with
\begin{equation}
	\label{elliptic-regularity}
	\norm{\psi}_{H^{1+\alpha}(\Domain)}
	\leq
	c \normLeb{U_\Mesh - \LebPro_\Mesh u}{\Domain}.
\end{equation} 

%
As a preliminary step, we derive a supercloseness estimate on the $L^2$-norm of $U_\Mesh - \LebPro_\Mesh u$. Unlike \cite{DiPietro.Ern.Lemaire:14}, we do not need to address the lowest-order case $\Degree = 0$ separately.

\begin{lemma}[Supercloseness $L^2$-estimate]
	\label{L:supercloseness}
	Let $u \in \SobH{\Domain}$ solve \eqref{Poisson} and denote by $\hU \in \HHOspace{}$ the solution of \eqref{HHO-method-modified} with $\HHOsmt$ as in \eqref{smoother-moment-preserving}. Let $\alpha \in (\frac12, 1]$ be such that \eqref{elliptic-regularity} is satisfied. Then, the following holds true with $h := \max_{K \in \Mesh} h_K$:
	\begin{equation}
		\label{L2-supercloseness}
		\normLeb{U_\Mesh - \LebPro_\Mesh u}{\Domain}^2
		\leq C_{\Shape, \Degree}
		h^{2\alpha} \sum_{K\in \Mesh} 
		\inf_{q \in \Poly{\Degree+1}{K}} \normLeb{\Grad(u-q)}{K}^2.
	\end{equation} 
\end{lemma}

\begin{proof}
We test \eqref{duality} with $U_\Mesh - \LebPro_\Mesh u$ and integrate by parts element-wise, exploiting the regularity of $\psi$. We obtain
\begin{equation*}
	\begin{split}
		&\normLeb{U_\Mesh - \LebPro_\Mesh u}{\Domain}^2
		=
		-\sum_{K,F} \int_F (U_\Mesh- \LebPro_\Mesh u) \Grad \psi \cdot \Normal_K
		+
		\int_\Domain \GradM (U_\Mesh- \LebPro_\Mesh u) \cdot \Grad \psi\\
		& \quad =
		-\sum_{K,F} \int_F ((U_\Mesh-U_\Skeleton)- (\LebPro_\Mesh u-\LebPro_\Skeleton u)) \Grad \psi \cdot \Normal_K
		+
		\int_\Domain \GradM (U_\Mesh- \LebPro_\Mesh u) \cdot \Grad \psi
	\end{split}
\end{equation*}
where $K$ and $F$ vary in $\Mesh$ and $\Faces{K}$, respectively. The second identity follows from the observation that each interface of $\Mesh$ appears twice in the sum, with opposite orientations, combined with the fact that both $u$ and $\hU$ vanish on the boundary faces. Note also that, to alleviate the notation, we write $U_\Mesh$ and $\LebPro_\Mesh u$, instead of $(U_\Mesh)_{|K}$ and $(\LebPro_\Mesh u)_{|K}$, in the face integrals. Owing to the definition of the reconstruction $\PotRec$ in \eqref{gradient-reconstruction}, we have
\begin{equation*}
	\begin{split}
		\int_\Domain \GradM \PotRec (\hU-\Interp u) \cdot \GradM \EllPro \psi
		=&
		-\sum_{K,F} \int_F ((U_\Mesh-U_\Skeleton)- (\LebPro_\Mesh u-\LebPro_\Skeleton u)) \Grad \EllPro \psi \cdot \Normal_K + \\
		& + 
		\int_\Domain \GradM (U_\Mesh - \LebPro_\Mesh u)
		\cdot 
		\GradM \EllPro \psi
	\end{split}
\end{equation*} 
because $\EllPro \psi \in \Poly{\Degree+1}{\Mesh}$ is an admissible test function in \eqref{gradient-reconstruction-grad}. Inserting this identity into the previous one and exploiting the $H^1$-orthogonality of the broken elliptic projection, we infer that
\begin{equation*}
	\begin{split}
		\normLeb{U_\Mesh - \LebPro_\Mesh u}{\Domain}^2
		=&
		-\sum_{K,F} \int_F ((U_\Mesh-U_\Skeleton)- (\LebPro_\Mesh u-\LebPro_\Skeleton u)) \Grad (\psi-\EllPro \psi) \cdot \Normal_K +\\
		& +
		\int_\Domain \GradM \PotRec (\hU-\Interp u) \cdot \GradM \EllPro \psi.
	\end{split}
\end{equation*}
In order to rewrite the second summand in the right-hand side, we recall the identity $\PotRec \circ \Interp = \EllPro$ from \eqref{RI=E}. Then, we exploit problems \eqref{Poisson} and \eqref{HHO-method-modified} and observe that the combination of Lemma~\ref{L:moment-preservation} with Proposition~\ref{P:smoother-moment-preserving} guarantees the validity of \eqref{algebraic-consistency-suffcond} which we exploit here as follows:
\begin{equation*}
\aextHHO(\Interp u,\hHHOsmt \Interp \psi) = \aextHHO(\Interp u,\Interp\psi) - \Stab(\Interp u,\Interp\psi).
\end{equation*}
Thus, we have $\int_\Domain \GradM \EllPro u \cdot \Grad \HHOsmt \Interp \psi  = \int_\Domain \GradM \PotRec \Interp u \cdot \GradM \EllPro \psi$, whence we infer that
\begin{equation}
	\label{L2-estimate-derivation1}
	\int_\Domain \GradM \PotRec (\hU-\Interp u) \cdot \GradM \EllPro \psi
	=
	-\Stab(\hU, \Interp \psi)
	+
	\int_\Domain \GradM (u-\EllPro u) \cdot \Grad \HHOsmt \Interp \psi .
\end{equation}
Therefore, exploiting again the $H^1$-orthogonality of $\EllPro$, we obtain
\begin{equation}
\label{L2-estimate-derivation2}
\normLeb{U_\Mesh - \LebPro_\Mesh u}{\Domain}^2
=
\mathfrak{T}_1 + \mathfrak{T}_2 + \mathfrak{T}_3
\end{equation}
with
\begin{gather}
\nonumber
\mathfrak{T}_1 
:= 
-\sum_{K,F} \int_F ((U_\Mesh-U_\Skeleton)- (\LebPro_\Mesh u-\LebPro_\Skeleton u)) \Grad (\psi-\EllPro \psi) \cdot \Normal_K\\
\nonumber
\mathfrak{T}_2 
:= 	
- \Stab(\hU, \Interp \psi)
\qquad  \qquad
\mathfrak{T}_3 := \int_\Domain \GradM  (u-\EllPro u) \cdot \GradM (\HHOsmt \Interp \psi  - \EllPro \psi).
\end{gather}
It remains to bound the three summands $\mathfrak{T}_1$, $\mathfrak{T}_2$ and $\mathfrak{T}_3$. The definition of the interpolant $\Interp$ and the coercivity stated in Lemma~\ref{L:discrete-stability} entail
	\begin{equation*}
		\mathfrak{T}_1^2
		\Cleq
		\normHHO{\hU-\Interp u}^2
		\sum_{K,F} h_F \normLeb{\Grad (\psi-(\EllPro \psi)_{|K})}{F}^2
	\end{equation*}
	where $K$ and $F$ vary in $\Mesh$ and $\Faces{K}$, respectively. Owing to the approximation properties of the broken elliptic projection, we obtain
	\begin{equation*}
	h_F \normLeb{\Grad(\psi-(\EllPro\psi)_{|K})}{F}^2 
	\Cleq
	h_K^{2\alpha} \normSemi{\psi}_{H^{1+\alpha}(K)}^2,
	\end{equation*}
	for all $K \in \Mesh$ and $F \in \Faces{K}$. Combining this bound and the previous one with the first part of Lemma~\ref{L:interpolation-errors} and the $H^1$-norm error estimate \eqref{H1-estimate-nonasymptotic}, we obtain
	\begin{equation*}
		\label{L2-estimate-T1}
		\mathfrak{T}_1^2
		\Cleq
		h^{2\alpha} \normSemi{\psi}_{H^{1+\alpha}(\Domain)}^2
		\sum_{K \in \Mesh} \inf_{q \in \Poly{\Degree+1}{K}}
		\normLeb{\Grad(u-q)}{K}^2.
	\end{equation*}
	Invoking again Lemma~\ref{L:interpolation-errors} and \eqref{H1-estimate-nonasymptotic} yields also 
	\begin{equation*}
		\label{L2-estimate-T2}
		\begin{split}
		\mathfrak{T}_2^2
		&\Cleq
		\left( \sum_{K \in \Mesh} \inf_{q \in \Poly{\Degree+1}{K}}
		\normLeb{\Grad(\psi-q)}{K}^2\right)
		\left( \sum_{K \in \Mesh} \inf_{q \in \Poly{\Degree+1}{K}}
		\normLeb{\Grad(u-q)}{K}^2\right)\\
		&\Cleq
		h^{2\alpha} \normSemi{\psi}_{H^{1+\alpha}(\Domain)}^2
		\sum_{K \in \Mesh} \inf_{q \in \Poly{\Degree+1}{K}}
		\normLeb{\Grad(u-q)}{K}^2,
		\end{split} 
	\end{equation*}	
	where the second estimate follows from standard polynomial approximation properties in Sobolev spaces. In order to bound the third summand $\mathfrak{T}_3$ in \eqref{L2-estimate-derivation2}, we proceed similarly to the proof of \eqref{smoother-moment-preserving-stability} in Proposition~\ref{P:smoother-moment-preserving}. Owing to the approximation properties of the broken elliptic projection, we only need to bound $\normLeb{\Grad(\EllPro \psi - \HHOsmt \Interp \psi)}{K}$. For all $K \in \Mesh$, the triangle inequality yields
	\begin{equation}
	\label{L2-estimate-derivation3}
	\normLeb{\Grad(\EllPro \psi - \HHOsmt \Interp \psi)}{K}
	\leq
	\normLeb{\Grad(\EllPro \psi - \AvgOper{}\Interp \psi)}{K}
	+
	\normLeb{\Grad\BubbOper{} (\Interp \psi - \hAvgOper{}\Interp \psi)}{K}.
	\end{equation} 
	The definitions of $\Interp$ and $\BubbOper{}$ readily imply that $\BubbOper{} \Interp \psi = \BubbOper{} \hpsi$, with $\hpsi = (\psi, (\psi)_{|\Skeleton})$. This observation, Proposition~\ref{P:bubble-smoother} and the multiplicative trace inequality \eqref{trace-inequality} yield
	\begin{equation*}
	\normLeb{\Grad\BubbOper{} (\Interp \psi - \hAvgOper{}\Interp \psi)}{K}
	\Cleq h_K^{-1} \normLeb{\psi - \AvgOper{}\Interp \psi}{K}
	+
	\normLeb{\Grad(\psi - \AvgOper{} \Interp \psi)}{K}.
	\end{equation*}
	Next, we combine \cite[Lemma~4.3]{Ern.Guermond:17} with the identity \eqref{RI=E} and the multiplicative trace inequality \eqref{trace-inequality}. We obtain that
	\begin{equation*}
	\begin{split}
	&h_K^{-1}\normLeb{\EllPro \psi - \AvgOper{} \Interp \psi}{K} + \normLeb{\Grad(\EllPro \psi - \AvgOper{}\Interp \psi)}{K}
	\Cleq
	\sum_{F \cap K \neq \emptyset}
	h_F^{-\frac{1}{2}} 
	\normLeb{\Jump{\EllPro \psi}}{F}\\
	&\qquad \Cleq
	\sum_{K' \cap K \neq \emptyset} \left( 
	h_{K'}^{-1} \normLeb{\psi - \EllPro\psi}{K'}
	+
	\normLeb{\Grad(\psi - \EllPro \psi)}{K'}
	\right)
	\end{split}
	\end{equation*}
	where $F$ and $K'$ vary in $\Faces{}$ and $\Mesh$, respectively, and $\Jump{\cdot}$ is the jump operator. We insert this inequality and the previous one into \eqref{L2-estimate-derivation3}. Owing to the approximation properties of the broken elliptic projection, we infer that
	\begin{equation*}
	\normLeb{\Grad(\EllPro \psi - \HHOsmt \Interp \psi)}{K}
	\Cleq 
	\sum_{K' \cap K \neq \emptyset} h_{K'}^{\alpha} \normSemi{\psi}_{H^{1+\alpha}(K')}.
	\end{equation*}
	Squaring and summing over all $K \in \Mesh$, we finally derive that 
	\begin{equation}
	\label{L2-estimate-T3}
	\mathfrak{T}_3^2
	\Cleq
	h^{2\alpha} \normSemi{\psi}_{H^{1+\alpha}(\Domain)}^2
	\sum_{K \in \Mesh} \inf_{q \in \Poly{\Degree+1}{K}}
	\normLeb{\Grad(u-q)}{K}^2.
	\end{equation}
	in view of \eqref{elliptic-projection-orthogonality} and recalling that the maximum number of simplices touching a given simplex is $\leq C_{\Shape}$.   
Collecting the bounds on $\mathfrak{T}_1$, $\mathfrak{T}_2$ and $\mathfrak{T}_3$ and invoking the elliptic regularity property \eqref{elliptic-regularity} concludes the proof.
\end{proof}


\begin{theorem}[$L^2$-norm error estimate]
	\label{T:L2-estimate}
	Let $u \in \SobH{\Domain}$ solve \eqref{Poisson} and denote by $\hU \in \HHOspace{}$ the solution of \eqref{HHO-method-modified} with $\HHOsmt$ as in \eqref{smoother-moment-preserving}. Let $\alpha \in (\frac12, 1]$ be such that \eqref{elliptic-regularity} is satisfied. Then, the following holds true:
	\begin{equation}
		\label{L2-estimate-nonasymptotic}
		\normLeb{u-\PotRec \hU}{\Domain}^2
		\leq C_{\Shape, \Degree}
		h^{2\alpha} \sum_{K\in \Mesh} 
		\inf_{q \in \Poly{\Degree+1}{K}} \normLeb{\Grad(u-q)}{K}^2
	\end{equation} 
	where $h := \max_{K \in \Mesh} h_K$. Furthermore, if $u \in H^m(\Domain)$ with $m \in \{ 1,\dots, \Degree+2 \}$, we have
	\begin{equation}
		\label{L2-estimate-asymptotic}
		\normLeb{u-\PotRec \hU}{\Domain}^2
		\leq C_{\Shape, \Degree}
		h^{2\alpha} \sum_{K\in \Mesh}
		h_K^{2(m-1)} \normSemi{u}^2_{H^m(K)}. 
	\end{equation} 
\end{theorem}
%
\begin{proof}
We have $\normLeb{u-\PotRec \hU}{\Domain} \leq \normLeb{u-\EllPro u}{\Domain} + \normLeb{\PotRec \hU - \EllPro u}{\Domain}$. Concerning the first summand, the identity \eqref{RI=E} and the bound \eqref{interpolation-errors-L2} imply that
	\begin{equation}
		\label{elliptic-projection-optimal}
		\normLeb{u-\EllPro u}{\Domain}^2
		\leq 
		\sum_{K\in \Mesh} \left( \dfrac{h_K}{\pi}\right) ^2 \inf_{q \in \Poly{\Degree+1}{K}}
		\normLeb{\Grad (u-q)}{K}^2.
	\end{equation}
	Concerning the other summand, the identity \eqref{RI=E} implies $\PotRec \hU - \EllPro u = \PotRec(\hU - \Interp u)$. We fix any $K \in \Mesh$ and denote by $\fint_K$ the integral mean value on $K$. The identity \eqref{gradient-reconstruction-avg} and the Poincar\'{e}--Steklov inequality \eqref{poincare-inequality} yield
	\begin{equation*}
		\begin{split}
			\normLeb{\PotRec(\hU-\Interp u)}{K}
			&\leq
			\normLeb{\PotRec(\hU-\Interp u) - \textstyle\fint_K \PotRec(\hU-\Interp u)}{K} 
			+ 
			\normLeb{\textstyle\fint_K \PotRec(\hU-\Interp u)}{K}\\
			& \leq
			\pi^{-1} h_K \normLeb{\Grad \PotRec(\hU-\Interp u)}{K}
			+
			\normLeb{U_\Mesh-\LebPro_\Mesh u}{K}.
		\end{split}
	\end{equation*}
	Summing over all simplices of $\Mesh$ and using the first part of Lemma~\ref{L:interpolation-errors} and the $H^1$-norm error estimate \eqref{H1-estimate-nonasymptotic}, we obtain
	\begin{equation}
		\label{L2-estimate-partial}
		\normLeb{u-\PotRec \hU}{\Domain}^2
		\!\Cleq\!
		\sum_{K \in \Mesh}  \!\dfrac{h_K^2}{\pi^2} \inf_{q \in \Poly{\Degree+1}{K}}
		\normLeb{\Grad(u-q)}{K}^2
		+
		\normLeb{U_\Mesh-\LebPro_\Mesh u}{\Domain}^2.
	\end{equation}
Thus, we derive \eqref{L2-estimate-nonasymptotic} by inserting the bound~\eqref{L2-supercloseness} into \eqref{L2-estimate-partial}. Finally, the estimate \eqref{L2-estimate-asymptotic} follows from \eqref{L2-estimate-nonasymptotic} and standard polynomial approximation properties in Sobolev spaces.
\end{proof}

Similarly to Theorem~\ref{T:H1-estimate}, estimate \eqref{L2-estimate-nonasymptotic} holds under the minimal regularity $u \in \SobH{\Domain}$, and \eqref{L2-estimate-asymptotic} exploits only the element-wise regularity of $u$. Still, both estimates are more pessimistic than \eqref{interpolation-errors-L2} in Lemma~\ref{L:interpolation-errors}, even for $\alpha=1$, if $\Mesh$ is a graded mesh. This is a general drawback of the estimates derived via the Aubin--Nitsche duality argument. Perhaps, a better result could be obtained with the help of the technique recently devised in \cite{Georgoulis.Makridakis.Pryer:18, Makridakis:18}. 


\section{Polytopic meshes}
\label{S:polytopic-meshes}

Since the HHO methods in \cite{DiPietro.Ern:15,DiPietro.Ern.Lemaire:14} are not only defined for matching simplicial meshes of $\Domain$, but more generally on polyhedral meshes possibly comprising hanging nodes, it is worth asking if we can relax the assumptions on $\Mesh$ in the previous sections. To this end, a first inspection reveals that the abstract results of section~\ref{SS:quasi-optimality} on the quasi-optimality of \eqref{HHO-method-modified} build only on the notion of interface and on the nondegeneracy of $\HHObform$. Of course, both ingredients are in any case needed in the definition of the space $\HHOspace{}$ and for the solution of problem \eqref{HHO-method-standard}. Thus, in principle, it appears possible to design HHO methods, that are quasi-optimal in the semi-norm $\normHHO{\cdot}$, within a larger class of \textit{polytopic} meshes.   

Proceeding as in \cite{DiPietro.Ern:15,DiPietro.Ern.Lemaire:14}, we now consider meshes $\Mesh = (K)_{K \in \Mesh}$ of $\Domain$ such that
\begin{itemize}
\item $\overline{\Domain} = \bigcup_{K \in \Mesh} \overline{K}$ and the cardinality of $\Mesh $ is finite,
\vspace{1mm} 
\item each cell $K \in \Mesh$ is an open polygon/polyhedron,
\vspace{1mm}
\item for all cells $K_1, K_2 \in \Mesh$ with $K_1 \neq K_2$, we have $K_1 \cap K_2 = \emptyset$.
\end{itemize} 

We say that $F \subset \overline{\Domain}$ is a face of $\Mesh$ if it is a subset, with nonempty relative interior, of some $(\Dim-1)$-dimensional affine space $H_F$ and if one of the following conditions holds true: either there are two distinct cells $K_1, K_2 \in \Mesh$ so that $F = \overline{K}_1 \cap \overline{K}_2 \cap H_F$ or there is one cell $K \in \Mesh$ so that $F = \overline{K} \cap \partial \Domain \cap H_F$. We collect in the set $\FacesMint$ all the interfaces, i.e. the faces of $\Mesh$ fulfilling the first condition.   
 
To preserve the validity of the results in section~\ref{SS:stability-interpolation}, we further assume that $\Mesh$ is an admissible mesh in the sense of \cite[Section~1.4]{DiPietro.Ern:12}. More precisely, we require that there is a matching simplicial submesh $\Submesh = (T)_{T \in \Submesh}$ of $\Mesh$, such that
\begin{itemize}
\item for each simplex $T \in \Submesh$, there is a cell $K \in \Mesh$ such that $T \subseteq K$ and $h_K\lesssim h_T$.
\end{itemize} 
The inequalities stated in \eqref{discrete-trace-inequality}, \eqref{trace-inequality} and \eqref{poincare-inequality} as well as the ones in Lemmata~\ref{L:discrete-stability} and \ref{L:interpolation-errors} still hold true under this assumption, possibly up to more pessimistic constants, depending on the shape regularity of $\Mesh$ and $\Submesh$. We refer to \cite[section~1.4]{DiPietro.Ern:12} and \cite{DiPietro.Ern.Lemaire:14} for a more detailed discussion on this point.    

The real bottleneck in the extension of our previous results is the construction of a smoother $\HHOsmt$, generalizing the one in Proposition~\ref{P:smoother-moment-preserving}. For this purpose, one option is to still write $\HHOsmt$ as the combination of a bubble smoother, which accommodates the conservation of the moments prescribed by Proposition~\ref{P:quasi-optimality-HHO}, and an averaging operator, that serves to keep under control the constant $C_\HHO$ in \eqref{consistency-constant}. 

For the sake of completeness, we sketch a possible construction for arbitrary $\Degree \geq 0$. For all $K \in \Mesh$, we can find a simplex $T_K \in\Submesh$ such that $T_K \subseteq K$. Denote by $\Phi_{T_K} \in \SobH{\Domain}$ the \textit{cell bubble} determined by $(i)$ $\Phi_{T_K} \equiv 0$ in $\overline{\Domain} \setminus T_K$, $(ii)$ $(\Phi_{T_K})_{|T_K} \in \Poly{\Dim + 1}{T_K}$ and $(iii)$ $\Phi_{T_K}(m_{T_K}) = 1$ at the barycenter $m_{T_K}$ of $T_K$. Since $(q_1, q_2) \mapsto \int_K q_1 q_2 \Phi_{T_K}$ is a scalar product on $\Poly{\Degree-1}{K}$, we define the operators $\BubbOper{K}$ and $\BubbOper{\Mesh}$ as in \eqref{smoother-cell} and \eqref{smoother-mesh}, respectively, with $\Phi_{T_K}$ in place of $\Phi_K$. 

For all $F \in \FacesMint$, we can find an interface $T_F$ of $\Submesh$ and $T_1, T_2 \in \Submesh$ so that $T_F \subseteq F$ and $T_F = T_1 \cap T_2$. Set $\omega_{T_F} := T_1 \cup T_2$ and denote by $\Phi_{T_F} \in \SobH{\Domain}$ the \textit{face bubble} obtained prescribing $(i)$ $\Phi_{T_F} \equiv 0$ in $\overline{\Domain} \setminus \omega_{T_F}$, $(ii)$ $(\Phi_{T_F})_{|T_j} \in \Poly{\Dim}{T_j}$ for $j=1,2$ and $(iii)$ $\Phi_{T_F}(m_{T_F}) = 1$ at the barycenter $m_{T_F}$ of $T_F$. We define the operator $\BubbOper{F}$ as in \eqref{smoother-face}, with $\Phi_{T_F}$ in place of $\Phi_F$. Then, for all $v_\Skeleton \in \Leb{\Skeleton}$, we set 
\begin{equation*}
\BubbOper{\Skeleton} v_\Skeleton
:=
\begin{cases}
\sum \limits_{F \in \FacesMint}
(\BubbOper{F} v_\Skeleton)\Phi_{T_F}, & \Degree =0,\\[3mm]
\sum \limits_{F \in \FacesMint}
\sum\limits_{z \in \Lagr{\Degree}{\Submesh} \cap T_F} (\BubbOper{F} v_\Skeleton)(z) \Phi_z \Phi_{T_F}, & \Degree \geq 1,
\end{cases}
\end{equation*}
where $\Lagr{\Degree}{\Submesh}$ denotes the Lagrange nodes of degree $\Degree$ of $\Submesh$ and $\Phi_z$ is the Lagrange basis function of $\Sob{\Domain} \cap \Poly{\Degree}{\Submesh}$ associated with the evaluation at $z$. 

With $\BubbOper{\Mesh}$ and $\BubbOper{\Skeleton}$ as indicated, the bubble smoother $\BubbOper{}: \Leb{\Domain} \times\Leb{\Skeleton} \to \SobH{\Domain}$ is simply given by \eqref{bubble-smoother} and fulfills \eqref{moment-preservation} and \eqref{bubble-smoother-stability}.

Finally, denote by $\LagrInt{\Degree+1}{\Submesh}$ the interior Lagrange nodes of degree $\Degree+1$ of $\Submesh$. For all $\hsigma = (\sigma_\Mesh, \sigma_\Skeleton) \in \HHOspace{}$, we consider the averaging 
\begin{equation}
\AvgOper{} \hsigma
:=
\sum_{z \in \LagrInt{\Degree+1}{\Submesh}}
\left( \dfrac{1}{\#\omega_z}
\sum_{T \in \omega_z} (\PotRec\hsigma)_{|T}(z) 
\right) 
\Phi_z, 
\end{equation}
where $T$ varies in $\Submesh$ and $\omega_z$ collects the simplices of $\Submesh$ to which $z$ belongs. 

With $\AvgOper{}$ and $\BubbOper{}$ as indicated, the smoother $\HHOsmt: \HHOspace{} \to \SobH{\Domain}$, defined as in Proposition~\ref{P:smoother-moment-preserving}, fulfills \eqref{moment-preservation} and \eqref{smoother-moment-preserving-stability}. The derivation of $H^1$- and $L^2$-norm error estimates of the HHO method \eqref{HHO-method-modified} with this smoother proceeds along the same lines as in section~\ref{SS:error-estimates}.

\begin{remark}[Use of the submesh]
\label{R:use-of-submesh}
The use of the simplicial submesh $\Submesh$ in the definition of the bubble smoother $\BubbOper{}$ is not really necessary. Indeed, one only needs bubble functions attached to the cells and to the interfaces of $\Mesh$ and bounded extension operators from each interface to $\Domain$. In contrast, our construction of the averaging $\AvgOper{}$ substantially builds on the submesh. Of course, this can be seen as a main disadvantage, as it restricts applicability of the proposed method to the class of admissible meshes. Still, it must be said that this is just one possible construction and that the use of alternative averaging operators could be further explored.
\end{remark}



\end{document}